\documentclass[11pt]{article}
\usepackage{amssymb,latexsym,amsmath,amsbsy,amsthm,amsxtra,amsgen, graphicx}
\newfont{\msbm}{msbm10 at 11pt}
\newcommand {\R} {\mbox{\msbm R}}

\def\cF{\mathcal{F}}

\newtheorem{Theo}{Theorem}
\newtheorem{Lemma}[Theo]{Lemma}

\newtheorem{Prop}[Theo]{Proposition}

\newtheorem{Rmk}[Theo]{Remark}

\DeclareMathOperator{\E}{\mathbb{E}}
\renewcommand{\P}{}
\let\P\relax
\DeclareMathOperator{\P}{\mathbb{P}}

\begin{document}
\title{Critical branching Brownian motion with absorption: \\ particle configurations}

\author{Julien Berestycki\thanks{Research supported in part by ANR-08-BLAN-0220-01 and ANR-08-BLAN-0190 and ANR-09-BLAN-0215}, Nathana\"el Berestycki\thanks{Research supported in part by EPSRC grants EP/GO55068/1 and EP/I03372X/1}
  \ and Jason Schweinsberg\thanks{Supported in part by NSF Grants DMS-0805472 and DMS-1206195}}
\maketitle



\vspace{-.6in}
\begin{abstract}
We consider critical branching Brownian motion with absorption, in which there is initially a single particle at $x > 0$, particles move according to independent one-dimensional Brownian motions with the critical drift of $-\sqrt{2}$, and particles are absorbed when they reach zero.  Here we obtain asymptotic results concerning the behavior of the process before the extinction time, as the position $x$ of the initial particle tends to infinity.  We estimate the number of particles in the system at a given time and the position of the right-most particle.  We also obtain asymptotic results for the configuration of particles at a typical time.
\end{abstract}

\section{Introduction}

We consider branching Brownian motion with absorption.  At time zero, there is a single particle at $x > 0$.  Each particle moves independently according to one-dimensional Brownian motion with a drift of $-\mu$, and each particle independently splits into two at rate $1$.  Particles are absorbed when they reach the origin.  This process was first studied in 1978 by Kesten \cite{kesten}, who showed that with positive probability there are particles alive at all times if $\mu < \sqrt{2}$, but all particles are eventually absorbed almost surely if $\mu \geq \sqrt{2}$.

In recent years, there has been a surge of renewed interest in this process.  Some of this interest has been driven by connections between branching Brownian motion with absorption and the FKPP equation.  See, for example, the work of Harris, Harris, and Kyprianou \cite{hhk06}, who used branching Brownian motion with absorption to establish existence and uniqueness results for the FKPP traveling-wave equation.  In other work, such as \cite{bbs, bdmm1, bdmm2, maillard}, branching Brownian motion with absorption or a very similar process has been used to model a population undergoing selection.  In this setting, particles represent individuals in a population, branching events correspond to births, the positions of the particles are the fitnesses of the individuals, and absorption at zero models the death of individuals whose fitness becomes too low.

In this paper, we consider branching Brownian motion with absorption in the critical case with $\mu = \sqrt{2}$.  This process is known to die out with probability one, but we are able to use techniques developed in \cite{bbs, bbs2} to obtain some new and rather precise results about the behavior of the process before the extinction time.  We focus on asymptotic results about the number of particles, the position of the right-most particle, and the configuration of particles as the position $x$ of the initial particle tends to infinity.

\subsection{Main results}

Let $N(s)$ be the number of particles at time $s$, and let $X_1(s) \geq X_2(s) \geq \dots \geq X_{N(s)}(s)$ denote the positions of the particles at time $s$.
Let
\begin{equation}\label{Ydef}
Y(s) = \sum_{i=1}^{N(s)} e^{\sqrt{2} X_i(s)}.
\end{equation}
Throughout the paper, we will use the constants
\begin{equation}\label{cdef}
\tau = \frac{2 \sqrt{2}}{3 \pi^2}, \hspace{.5in}c = \tau^{-1/3} = \bigg( \frac{3 \pi^2}{2 \sqrt{2}} \bigg)^{1/3}.
\end{equation}
Let $t = \tau x^3$, which is approximately the extinction time of the process when $x$ is large.  More precisely, it was shown in \cite{bbs2} that that for all $\varepsilon > 0$, there is a positive constant $\beta$ such that for sufficiently large $x$, the extinction time is between $t - \beta x^2$ and $t + \beta x^2$ with probability at least $1 - \varepsilon$.

Our first result shows how the number of particles evolves over time.  For times $s$ between $Bx^2$ and $(1 - \delta)t$, where $B$ is a large constant and $\delta$ is a small constant, with high probability this result estimates the number of particles at time $s$ to within a constant factor.

\begin{Theo}\label{numthm}
Fix $\varepsilon > 0$ and $\delta > 0$.  Then there exists a positive constant $B$ depending on $\varepsilon$ and positive constants $C_1$ and $C_2$ depending on $B$, $\delta$, and $\varepsilon$ such that for sufficiently large $x$, we have
$$\P \bigg( \frac{C_1}{x^3} e^{\sqrt{2} (1 - s/t)^{1/3}x} \leq N(s) \leq \frac{C_2}{x^3} e^{\sqrt{2} (1 - s/t)^{1/3}x} \bigg) > 1 - \varepsilon$$
for all $s \in [Bx^2, (1 - \delta) t]$.
\end{Theo}

For $0 \leq s \leq t$, define
\begin{equation}\label{Ldef}
L(s) = x \bigg( 1 - \frac{s}{t} \bigg)^{1/3}  = c(t - s)^{1/3}.
\end{equation}
The next result shows that at time $s$, the right-most particle is usually slightly to the left of $L(s)$.
\begin{Theo}\label{rtthm}
Suppose $0 < u < \tau$, and let $s = u x^3$.  Let $\varepsilon > 0$.  Then there exist $d_1 > 0$ and $d_2 > 0$, depending on $u$ and $\varepsilon$, such that for sufficiently large $x$, $$\P\bigg(L(s) - \frac{3}{\sqrt{2}} \log x - d_1 < X_1(s) < L(s) - \frac{3}{\sqrt{2}} \log x + d_2 \bigg) > 1 - \varepsilon.$$
\end{Theo}
We are also able to obtain results about the entire configuration of particles.  The key idea is that at time $s$, the density of particles near $y \in (0, L(s))$ will be roughly proportional to
\begin{equation}
e^{-\sqrt{2} y} \sin \bigg( \frac{\pi y}{L(s)} \bigg). \label{density}
\end{equation}
Establishing a rigorous version of this statement requires proving two theorems.  In Theorem \ref{config1}, we consider the probability measure in which a mass of $1/N(s)$ is placed at the position of each particle at time $s$.  Because most particles are close to the origin and $\sin(\pi y/L(s)) \approx \pi y/L(s)$ for small $y$, in the limit this probability measure has a density proportional to $y e^{-\sqrt{2} y}$.  In Theorem \ref{config2}, we consider the probability measure in which a particle at position $z$ is assigned a mass proportional to $e^{\sqrt{2} z}$.  In this case, particles over the entire interval from $0$ to $L(s)$ contribute significantly even in the limit, and the sinusoidal shape is observed.

For these results, we use $\Rightarrow$ to denote convergence in distribution for random elements in the Polish space of probability measures on $(0, \infty)$, endowed with the weak topology.  We also use $\delta_y$ to denote a unit mass at $y$.

\begin{Theo}\label{config1}
Suppose $0 < u < \tau$, and let $s = ux^3$.  Define the probability measure $$\chi(u) = \frac{1}{N(s)} \sum_{i=1}^{N(s)} \delta_{X_i(s)}.$$  Let $\mu$ be the probability measure on $(0, \infty)$ with density $g(y) = 2y e^{-\sqrt{2} y}$.  Then $\chi(u) \Rightarrow \mu$ as $x \rightarrow \infty$.
\end{Theo}

\begin{Theo}\label{config2}
Suppose $0 < u < \tau$, and let $s = ux^3$.  Define the probability measure $$\eta(u) = \frac{1}{Y(s)} \sum_{i=1}^{N(s)} e^{\sqrt{2} X_i(s)} \delta_{X_i(s)/L(s)}.$$  Let $\nu$ be the probability measure on $(0,1)$ with density $h(y) = \frac{\pi}{2} \sin(\pi y)$.  Then $\eta(u) \Rightarrow \nu$ as $x \rightarrow \infty$.
\end{Theo}

\subsection{Ideas behind the proofs}

In this section, we briefly outline a few of the heuristics behind the main results and their proofs.
One can not obtain an accurate estimate of the number of particles $N(s)$ at time $s$ simply by calculating the expected value $E[N(s)]$, as the expected value is dominated by rare events when the number of particles is unusually large.  Instead, we use the method of truncation and obtain our estimates by calculating first and second moments for a process in which particles are killed if they get too far to the right.  It is useful to consider first branching Brownian motion with a drift of $-\sqrt{2}$ in which particles are killed when they reach either $0$ or $L$.  For this process, if we start with a single particle at $x$ and if $s$ is large, then the ``density" of particles near $y$ at time $s$ is approximately (see Lemma \ref{stripdensity} below)
\begin{equation}\label{psnew}
p_s(x,y) = \frac{2}{L} e^{-\pi^2 s/2L^2} e^{\sqrt{2} x} \sin \bigg( \frac{\pi x}{L} \bigg) e^{-\sqrt{2} y} \sin \bigg( \frac{\pi y}{L} \bigg).
\end{equation}
This formula indicates that in the long run, the particles settle into an equilibrium configuration in which the density of particles near $y$ is proportional to $e^{-\sqrt{2}y} \sin(\pi y/L)$.

We need to choose $L$ to be large enough that particles are unlikely to be killed at $L$, but small enough that we can obtain useful moment bounds.  This will require choosing $L$ to be near the position of the right-most particle.  However, because of the exponential decay term in (\ref{psnew}), the number of particles is decreasing over time, and therefore the position of the right-most particle will decrease also.  Consequently, it will be necessary to allow the position of the right boundary to decrease over time and consider a process in which particles are killed if they reach $L(s)$ at time $s$.

It turns out that the good choice for the killing boundary is given by $L(s)=c(t-s)^{1/3}$, as in (\ref{Ldef}) above.  The importance 
of this position was already recognized by Kesten \cite{kesten}, where it plays a key role in his analysis of the critical branching Brownian motion with absorption.  In \cite{bbs2}, we followed a similar strategy to study the survival probability of the same process and we showed that a particle that reaches $L(s)$ at some time $s<t$ has a good probability to have a descendent alive at time $t$. A heuristic derivation of the precise form of $L(s)$ is given in section \ref{L(s)section}

More precisely, we are able to show that the probability that a particle hits $L(s)$ during a fixed time interval of length $O(x^2)$ (see Lemma \ref{fphi}), or that a particle ever hits a barrier $L_\alpha(s) =L(s) +O(\alpha)$ (see Lemma \ref{killalpha})
is close to 0 for $x$ and $\alpha$ large enough. The upshot of these results is that we can now choose to kill particles at $L(s)$ during the appropriate time frame in addition to killing them at 0 at no additional cost. This allows us to use and refine previous estimates and results from \cite{bbs,bbs2}  for the branching Brownian motion with absorption at 0 and at a fixed $L>0$ 
(see section \ref{constbdrysec}) or at a curved boundary $L(s)$ as in \eqref{Ldef} (see section \ref{curvedbdry}).

We now describe more precisely the structure of the proof and how the different estimates are used to obtain the proofs of Theorems \ref{numthm}, \ref{config1}, and \ref{config2}. 



As can be seen from (\ref{psnew}), 
in the model with killing at 0 and $L(s)$,  
the number of particles that will be in a given set at a sufficiently large future time is well predicted by the quantity $$Z(s) = \sum_{i=1}^{N(s)} e^{\sqrt{2} X_i(s)} \sin \bigg( \frac{\pi X_i(s)}{L(s)} \bigg) {\bf 1}_{\{X_i(s) \leq L(s)\}}.$$  Therefore $Z(s)$ is the natural measure of the ``size" of the process at time $s$. 

Given a bounded continuous function $f: [0,\infty) \to \R$ let us consider the sum 
$\sum_{i=1}^{N(s)} f(X_i(s)). $  We observe that $N(s)$ is this sum with $f\equiv 1$ and that to prove Theorem \ref{config1}, it suffices to show that for $s$ of the form $ux^3$, we have
$$
\frac1{N(s)} \sum_{i=1}^{N(s)} f(X_i(s)) \to_p \int_0^\infty g(y)f(y) dy ,
$$
where $g(y) = 2y e^{-\sqrt{2} y}$ for $y \geq 0$ and $\rightarrow_p$ denotes convergence in probability as $x \rightarrow \infty$.
In Lemma \ref{fphi}, we show that if $r=s-Bx^2$, where $B$ is a large constant, then with probability close to 1, no particle ever reaches $L(u)$ for $u \in [r,s]$, and therefore we can kill particles there at no cost during $[r,s].$ 
We denote by $X(f)$ the sum $\sum_{i=1}^{N(s)} f(X_i(s))$ when we do kill the particles at $L(u)$ for $u \in [r,s]$. 

On an interval of length $O(x^2)$, the boundary $L(u)$ stays roughly constant (it changes at most by a constant) and we are therefore able to use estimates obtained for the model in which particles are killed at 0 and at a fixed point $L$. More precisely, Lemma \ref{meanXf}  uses the estimates in Lemma \ref{fXvar} and Lemma \ref{stripdensity} to show that 
$$
\E[X(f) |\cF_r] \cong Z(r) {\pi \over L(s)^2} e^{-\pi^2 Bx^2 /2L(s)^2} \int_0^\infty f(y) g(y) dy, 
$$
where for this heuristic description we can take $\cong$ to mean ``is close to''.
In the same spirit, Lemma \ref{varXf} gives an upper bound of $\text{Var}(X(f) |\cF_r)$ by a quantity involving $Y(r).$

Focusing on the proof of Theorem 1, that is taking $f\equiv 1$ above, we see that we can plug the bounds on $\E[X(f) |\cF_r]$ and $\text{Var}(X(f) |\cF_r)$ into Chebyshev's inequality to show that $X(1)$ is not too far from its conditional expectation, as shown in equation \eqref{cheb} which says that
\begin{equation*}
\P \bigg(\big|X(1) - \E[X(1)|{\cal F}_r] \big| > \frac{1}{2} \E[X(1)|{\cal F}_r] \bigg| {\cal F}_r \bigg) \leq \frac{C Y(r) e^{\sqrt{2} L(s)}}{x^{3/2} {\hat Z}^2},
\end{equation*}
where ${\hat Z} \cong Z(r)$.
Propositions \ref{Zlower} and \ref{Zupper} show that on an event of probability close to 1 we have good bounds on $Z(r)$ while the upper bound for $Y(r)$ is given in Proposition \ref {Yupper}. The conclusion is that on an event of probability close to 1, the quantity on the right-hand side tends to zero as $x \rightarrow \infty$, and on the same event we can bound $\E[X(1)|\cF_r]$ as above.  This allows us to obtain the conclusion of Theorem \ref{numthm} because $X(1)=N(s)$ when no particle is killed at $L(u)$ during $[r,s]$.   



Propositions \ref{Zlower} and \ref{Zupper} are themselves obtained by considering the model where particles are killed at $L_\alpha(s)$ in addition to being killed at 0 as explained above. Using this idea, Lemma \ref{EZbound} bounds $\E[Z(s)|\cF_r] /Z(r)$ between two functions of $r$ and $s$.
Since Propositions \ref{Zlower} and \ref{Zupper} essentially derive from another application of Chebyshev's inequality, the bound on the second moment of $Z(s)$ given in Proposition \ref{VarZProp}   is a crucial step. 

Theorem \ref{config1} is obtained through a similar careful application of Chebyshev's inequality. Theorem \ref{config2} follows the same principle with 
$$
\sum_{i=1}^{N(s)} e^{\sqrt{2} X_i(s)} \phi\bigg( \frac{X_i(s)}{L(s)} \bigg) {\bf 1}_{\{X_i(s) < L(s)\}}.
$$
in lieu of 
$$
\sum_{i=1}^{N(s)} f(X_i(s)).
$$


The proof of Theorem \ref{rtthm} uses a different technique because the moment bounds obtained in this paper are not sharp enough near the right boundary $L(s)$ to give such precise control over the position of the right-most particle.  Instead, to control the position of the right-most particle at time $s$, we consider the configuration of particles at time $s - \gamma x^2$, where $\gamma$ is a small constant.  We estimate, for each particle at time $s - \gamma x^2$, the probability that it will have a descendant particle to the right of $L(s) - (3/\sqrt{2}) \log x + d$ at time $s$.  To estimate this probability, we can use a result of Bramson \cite{bram83} for the position of the right-most particle in branching Brownian motion without absorption, because it is unlikely that a particle that gets to the right of $L(s) - (3/\sqrt{2}) \log x + d$ would have hit zero between times $s - \gamma x^2$ and $s$.  The logarithmic term that appears in Theorem \ref{rtthm} is therefore closely related to the logarithmic correction in the celebrated result of Bramson, which says that for branching Brownian motion without drift or absorption, the median position of the right-most particle at time $t$ is within a constant of $$\sqrt{2} t - \frac{3}{2 \sqrt{2}} \log t.$$

\subsection{Heuristic derivation of $L(s)$ }\label{L(s)section}

There is a natural approximate relationship, discussed in \cite{bbs}, between the number of particles $N(s)$ and the optimum position of the right boundary $L(s)$ given by
\begin{equation}\label{LNrelation}
L(s) = \frac{1}{\sqrt{2}} (\log N(s) + 3 \log \log N(s)).
\end{equation}
Equation (\ref{psnew}) indicates that once the process is in equilibrium, the number of particles gets multiplied by $e^{-\pi^2 s/2 L(s)^2}$ after time $s$, which suggests the rough approximation $$N'(s) \approx - \frac{\pi^2}{2 L(s)^2} N(s).$$  Because we have written $L(s)$ as a function of $N(s)$, this approximation can be combined with the chain rule to give
\begin{equation}\label{Ldiffeq}
L'(s) \approx - \frac{1}{\sqrt{2} N(s)} \cdot \frac{\pi^2}{2 L(s)^2} N(s) = - \frac{\pi^2}{2 \sqrt{2} L(s)^2}.
\end{equation}
Because we begin with a particle at $x$, the position of the right-most particle will be near $x$ at small times, so we will take $L(0) = x$ (although this means that we will not be able to start killing particles at the right boundary immediately).  Solving the differential equation (\ref{Ldiffeq}) with $L(0) = x$ gives us $L(s) = c (t - s)^{1/3}$, where $t = \tau x^3$, as in (\ref{Ldef}).  Combining this with (\ref{LNrelation}) gives $$N(s) \approx \frac{e^{\sqrt{2} L(s)}}{(\log N(s))^3} \approx \frac{e^{\sqrt{2} L(s)}}{L(s)^3},$$ where $\approx$ means that the expressions on each side should be the same order of magnitude.  Because $x$ and $L(s)$ are of the same order of magnitude when $s \in [Bx^2, (1 - \delta) t]$, this approximation suggests the result of Theorem \ref{numthm}.

\subsection{Discussion and open problems}

Beyond the general motivation recalled at the beginning of this introduction, we now explain some of the contexts in which these results might be applied, and some related open problems that are raised by them.

\emph{Yaglom limit laws}. Let $x>0$ be fixed, and consider a branching Brownian motion with the critical drift of $-\sqrt{2}$ and absorption at zero started from one particle at $x$.  Almost surely, the process becomes extinct.  However, one can condition the process on survival up to a large time $t$: it is then interesting to consider, for example, the number of particles still alive at time $t$ or the configuration of particles at some time $s \in (0, t]$.  We believe that the results in Theorems \ref{config1} and \ref{config2} may be relevant to this question as well.  Similar questions for ordinary branching processes were considered by Yaglom \cite{yaglom}.  Note that the results of \cite{bbs2} give estimates, up to a multiplicative constant, for the probability of survival up to time $t$.

\emph{Fleming-Viot processes.} Critical branching Brownian motion with absorption shares several features with the Fleming-Viot process studied by Burdzy et al. \cite{burdzy1, burdzy2}, in the case where the underlying motion is $(X_t,t\ge0)$, a Brownian motion with drift $-\sqrt{2}$ and absorption at 0. This is a process with $N$ particles performing Brownian motion with drift $-\sqrt{2}$ and which are absorbed at 0. Furthermore, whenever a particle is absorbed at 0, a new particle is instantaneously born, at a location chosen uniformly among the set of the other $N-1$ other particles. For this Fleming-Viot process, the main question concerns the equilibrium empirical distribution of particles $\mu_N$ and in particular the limiting behaviour of $\mu_N$ as $N$ tends to infinity. Under fairly general conditions $\mu_N$ is conjectured to converge weakly to a particular quasi-stationary distribution of the underlying motion $X$. This has been verified in the case where the underlying motion $X$ is Brownian motion and absorption occurs at the boundary of a bounded domain $D$ (\cite{burdzy2, grigorescu1, grigorescu2}) or if the state space is finite (\cite{AFG}). In all these cases the quasi stationary distribution is unique. Recently, this question was also addressed in \cite{AFGJ} in the case where the underlying motion is a subcritical branching process on $\mathbb{Z}$. For this motion there is a continuum of quasi-stationary distributions, and the conclusion of \cite{AFGJ} is that the limiting behaviour of $\mu_N$ is given by the minimal one, in the sense of minimal expected absorption time. This \emph{selection principle} is also conjectured to hold in great generality. 

We point out that the function $g(x) = 2x e^{- \sqrt{2} x} $, which arises in Theorem \ref{config1} as the weak limit of the density appearing in \eqref{density} (and before that in \cite{bbs}) is a left eigenfunction of the generator of $X$,
$$
\frac12 \frac{d^2 }{ dx^2} - \sqrt{2} \frac{d}{dx},
$$ 
with Dirichlet boundary conditions on $(0, \infty)$. It is easily verified that this implies that $g$ is a quasi-stationary distribution of $X$. It is in fact the minimal quasi-stationary distribution of $X$ on $(0, \infty)$ in the sense of \cite{AFGJ}. See \cite{ServetMartinez} for precise calculations and relation to the corresponding Yaglom problem for $X$. Hence its appearance in Theorem \ref{config1} adds support to the above mentioned conjecture. In fact, our work raises the question of whether the function $e^{-\sqrt{2}x } \sin(\pi x / L)$ provides an even better approximation for $\mu_N$ when $L = (1/\sqrt{2}) ( \log N + 3 \log \log N)$. More formally, we ask whether the analogue of Theorem \ref{config2} also holds for $\mu_N$. 

\emph{Extreme configurations.} Theorem \ref{rtthm} determines the position of the right-most particle at time $s$ to within an additive constant.  A natural open problem is to get a convergence in distribution for the position of the rightmost particle. More generally, one can ask about the distribution of the particle configuration as seen from the rightmost particle, or from the median position of the rightmost particle. We point out that these questions also make sense in the nearly-critical case studied in \cite{bbs}, and that the proof of Theorem \ref{rtthm} can be adapted to that setting.  The analogous question for a branching Brownian motion without absorption at zero has been settled in \cite{abbs, abk}.

\section{Preliminary Estimates}\label{stripsec}

In this section, we obtain or recall some preliminary estimates concerning branching Brownian motion in which particles are killed not only at the origin but also when they travel sufficiently far to the right.  We will consider two cases.  One is when the Brownian particles are killed at some level $L > 0$.  The other is when particles are killed when they reach $L(s) = c (t-s)^{1/3}$ for some $s$.

As before, let $N(s)$ be the number of particles at time $s$, and denote the positions of the particles at time $s$ by $X_1(s) \geq X_2(s) \geq \dots \geq X_{N(s)}(s)$.  Define $Y(s)$ as in (\ref{Ydef}).  Let $({\cal F}_s, s \geq 0)$ denote the natural filtration associated with the branching Brownian motion.  Let $q_s(x,y)$ denote the density of the branching Brownian motion, meaning that if initially there is a single particle at $x$ and $A$ is a Borel subset of $(0, \infty)$, then the expected number of particles in $A$ at time $s$ is $$\int_A q_s(x,y) \: dy.$$

Here, and throughout the entire paper, $C$, $C'$, and $C''$ will denote positive constants whose value may change from line to line.  Numbered positive constants of the form $C_k$ will not change their values from line to line.

\subsection{A constant right boundary}\label{constbdrysec}

Let $L > 0$.  We consider here the case in which particles are killed upon reaching either $0$ or $L$.  This case was studied in detail in \cite{bbs}.  The following result is Lemma 5 of \cite{bbs}.

\begin{Lemma}\label{stripdensity}
For $s > 0$ and $x,y \in (0, L)$, let $$p_s(x,y) = \frac{2}{L} e^{-\pi^2 s/2L^2} e^{\sqrt{2} x} \sin \bigg( \frac{\pi x}{L} \bigg) e^{-\sqrt{2} y} \sin \bigg( \frac{\pi y}{L} \bigg),$$ and define $D_s(x,y)$ so that $$q_s(x,y) = p_s(x,y)(1 + D_s(x,y)).$$  Then for all $x,y \in (0,L)$, we have
\begin{equation}\label{Dineq}
|D_s(x,y)| \leq \sum_{n=2}^{\infty} n^2 e^{-\pi^2(n^2 - 1)s/2L^2}.
\end{equation}
\end{Lemma}

Lemma \ref{stripdensity} allows us to approximate $q_s(x,y)$ by $p_s(x,y)$ when $s$ is sufficiently large.  Lemma \ref{mainqslem} below collects some further results about the density $q_s(x,y)$.

\begin{Lemma} \label{mainqslem}
Fix a positive constant $b > 0$.  There exists a constant $C$ (depending on $b$) such that for all $s$ such that $s \geq b L^2$, we have
\begin{equation}\label{q1}
q_s(x,y) \leq C p_s(x,y), \, \, \forall x,y \in [0,L]
\end{equation}
and for all $s$ such that $s \leq b L^2$, we have
\begin{equation}\label{q2}
q_s(x,y) \leq \frac{C L^3}{s^{3/2}} p_s(x,y), \, \, \forall x,y \in [0,L].
\end{equation}
The following inequalities hold in general (for all $s>0$ and $x,y \in [0,L]$):
\begin{align}
q_s(x,y) &\leq \frac{C e^{\sqrt{2}(x - y)} e^{-(x - y)^2/2s}}{s^{1/2}} \label{q3} \\
\int_0^L q_s(x,y) \: dy &\leq e^s \label{q4} \\
\int_0^{\infty} q_s(x,y) \: ds &\leq \frac{2 e^{\sqrt{2}(x-y)} x (L - y)}{L} \label{q5} \\
\int_0^L e^{\sqrt{2} y} q_s(x,y) \: dy &\leq e^{\sqrt{2} x} \min \bigg\{1, \frac{L-x}{s^{1/2}} \bigg\} \label{q6}
\end{align}
\end{Lemma}

\begin{proof}
Equation (\ref{q1}) holds because the right-hand side of (\ref{Dineq}) is bounded by a constant when $s/L^2 \geq b$.  The result (\ref{q2}) is established in the proof of Proposition 14 in \cite{bbs} (see the argument between equations (53) and (54) of \cite{bbs}) by breaking the sum on the right-hand side of (\ref{Dineq}) into blocks of size approximately $L/\sqrt{s}$.  Equation (\ref{q3}) is equation (55) of \cite{bbs} and is obtained by comparing $q_s(x,y)$ to the density of standard Brownian motion at time $s$.  Equation (\ref{q4}) follows from the fact that the expected number of particles at time $s$ is at most $e^s$ because branching occurs at rate $1$.  Equation (\ref{q5}) follows from (28) and (51) of \cite{bbs} and is proved using Green's function estimates for Brownian motion in a strip.

Finally, to prove (\ref{q6}), let $v_s(x,y)$ be the density of Brownian motion killed at $0$ and $L$, meaning that if $A$ is a Borel subset of $(0,L)$, then the probability that a Brownian motion started at $x$ is in $A$ at time $s$ and has not hit $0$ or $L$ before time $s$ is $\int_A v_s(x,y) \: dy$.  By equation (28) of \cite{bbs}, we have
\begin{equation}\label{111}
q_s(x,y) = e^{\sqrt{2}(x-y)} v_s(x,y).
\end{equation}
Let $(B(t), t \geq 0)$ be standard Brownian motion with $B(0) = x$.  Then, by the Reflection Principle,
\begin{align}\label{112}
\int_0^L v_s(x,y) \: dy &= \P(B(t) \in (0, L) \mbox{ for all }t \in [0,s]) \nonumber \\
&\leq \P \big( \max_{0 \leq t \leq s} B(t) \leq L \big) \nonumber \\
&= 2 \int_0^{L-x} \frac{1}{\sqrt{2 \pi s}} e^{-y^2/2s} \: dy \nonumber \\
&\leq \min \bigg\{1, \frac{L-x}{s^{1/2}} \bigg\},
\end{align}
and (\ref{q6}) follows from (\ref{111}) and (\ref{112}).
\end{proof}

Let $$Z(s) = \sum_{i=1}^{N(s)} e^{\sqrt{2} X_i(s)} \sin \bigg( \frac{\pi X_i(s)}{L} \bigg).$$  Lemma \ref{YZexplem} below is part of Lemma 7 of \cite{bbs}, while Lemma \ref{varZlem} below follows immediately from Lemma 9 of \cite{bbs}.

\begin{Lemma}\label{YZexplem}
For all initial configurations of particles at time zero,
we have
\begin{equation}\label{Zexp}
\E[Z(s)] = e^{-\pi^2 s/2L^2} Z(0)
\end{equation}
and
\begin{equation}\label{Yexp}
\E[Y(s)] = \frac{4}{\pi} e^{-\pi^2s/2L^2} Z(0)(1 + D(s)),
\end{equation}
where $|D(s)|$ is bounded above by the right-hand side of (\ref{Dineq}).
\end{Lemma}

\begin{Lemma}\label{varZlem}
Fix a constant $b > 0$.  Suppose initially there is a single particle at $x$.  Then there exists a positive constant $C$, depending on $b$ but not on $L$ or $x$, such that for all $s \geq bL^2$, $$\E[Z(s)^2] \leq \frac{C e^{\sqrt{2} x} e^{\sqrt{2} L} s}{L^4}.$$
\end{Lemma}

The next lemma is Lemma 8 in \cite{bbs}, where it is obtained as a straightforward application of results in \cite{sawyer}.  It is also similar to Lemma 3.1 of \cite{kesten}.

\begin{Lemma}\label{fXvar}
Suppose $f:(0, L) \rightarrow [0, \infty)$ is a bounded measurable function.  Suppose initially there is a single particle at $x$.  Then $$\E \bigg[ \sum_{i=1}^{N(s)} f(X_i(s)) \bigg] = \int_0^L f(y) q_s(x,y) \: dy$$ and
$$\E \bigg[ \bigg( \sum_{i=1}^{N(s)} f(X_i(s)) \bigg)^2 \bigg] = \int_0^L f(y)^2 q_s(x,y) \: dy + 2 \int_0^s \int_0^L q_u(x,z) \bigg( \int_0^L f(y) q_{s-u}(z,y) \: dy \bigg)^2 \: dz \: du.$$
\end{Lemma}

\subsection{A curved right boundary}\label{curvedbdry}

Fix any time $t > 0$.  As in (\ref{Ldef}), for $s \in [0, t]$, let $$L(s) = c (t-s)^{1/3},$$ where $c$ was defined in (\ref{cdef}).  Consider branching Brownian motion with drift $-\sqrt{2}$ in which particles are killed if they reach zero, or if they reach $L(s)$ at time $s$.  Note that all particles must be killed by time $t$ because $L(t) = 0$.  This right boundary was previously considered by Kesten \cite{kesten}.  We recall here some results that recently appeared in \cite{bbs2}, where they were proved using techniques developed by Harris and Roberts \cite{haro}.

Let $$Z(s) = \sum_{i=1}^{N(s)} e^{\sqrt{2} X_i(s)} \sin \bigg( \frac{\pi X_i(s)}{L(s)} \bigg),$$ a quantity of crucial importance in what follows.  The next result, which combines Proposition 10 and Corollary 11 of \cite{bbs2}, provides a precise estimate of $\E[Z(s)]$.

\begin{Lemma}\label{EZbound}
For $0 < r < s < t$, let
\begin{equation}\label{Gdef}
G_r(s) = \exp \bigg( - (3 \pi^2)^{1/3} \big((t-r)^{1/3} - (t-s)^{1/3} \big) \bigg) \bigg( \frac{t-s}{t-r} \bigg)^{1/6}.
\end{equation}
There exist positive constants $C_3$ and $C_4$ such that if $0 < s < t$, then
$$Z(0) G_0(s) \exp(- C_3 (t-s)^{-1/3})  \le \E[Z(s)] \le Z(0) G_0(s) \exp(C_4 (t-s)^{-1/3})$$ and, more generally, if $0 < r < s < t$, then $$Z(r) G_r(s) \exp(-C_3 (t - s)^{-1/3}) \leq \E[Z(s)|{\cal F}_r] \leq Z(r) G_r(s) \exp(C_4(t - s)^{-1/3}).$$
\end{Lemma}

The following result, which is the $r = 0$ case of Proposition 12 in \cite{bbs2}, establishes bounds on the density up to a constant factor.

\begin{Lemma}\label{densityprop}
For $x,y>0$ and $0 \le s \le t$, let $$\psi_s(x,y) = \frac1{L(s)} e^{-(3 \pi^2)^{1/3}(t^{1/3} - (t-s)^{1/3})} \bigg( \frac{t-s}{t} \bigg)^{1/6} e^{\sqrt{2} x} \sin \bigg( \frac{\pi x}{L(0)} \bigg) e^{-\sqrt{2} y} \sin \bigg( \frac{\pi y}{L(s)} \bigg).$$
Fix a positive constant $b$.  There exists a constant $A > 0$ and positive constants $C'$ and $C''$, with $C''$ depending on $b$, such that if $L(0)^2 \leq s \leq t - A$, then $$q_s(x,y) \geq C' \psi_s(x,y)$$ and if $bL(0)^2 \leq s \leq t - A$, then $$q_s(x,y) \leq C'' \psi_s(x,y).$$
\end{Lemma}

We will also require estimates on the number of particles killed at the right boundary.  The result below is the $s = 0$ case of Lemma 15 in \cite{bbs2}.

\begin{Lemma} \label{L:sumRj}
Suppose there is initially a single particle at $x$, where $0 < x < L(0)$.  Let $R$ be the number of particles killed at $L(s)$ for some $s \in [0,t]$.  Then there are positive constants $C'$ and $C''$ such that $$C' h(x) \leq \E[R] \leq C''(h(x) + j(x)),$$ where $$h(x) = e^{\sqrt{2}x} \sin\left(\frac{\pi x}{ct^{1/3}}\right) t^{1/3} \exp(- (3\pi^2t)^{1/3})$$ and $$j(x) = x e^{\sqrt{2}x} t^{-1/3} \exp(- (3\pi^2t)^{1/3}).$$
\end{Lemma}

Finally, we will need the following bound on the second moment of $Z(s)$.

\begin{Prop}\label{VarZProp}
Fix $\kappa > 0$ and $\delta > 0$.  Then there exists a positive constant $C$, depending on $\kappa$ and $\delta$ but not on $t$, such that for all $t \geq 1$ and all $s$ satisfying $\kappa t^{2/3} \leq s \leq (1 - \delta) t$,
$$\textup{Var}(Z(s)) \leq C \E[Z(s)]^2 \bigg( \frac{e^{\sqrt{2} L(0)}}{L(0)Z(0)} + \frac{e^{\sqrt{2} L(0)} Y(0)}{L(0)^2 Z(0)^2} \bigg).$$
\end{Prop}

\begin{proof}
The proof is similar to the proof of Lemma 12 in \cite{bbs}.  Choose times $0 = s_0 < s_1 < \dots < s_K = s$ such that $\kappa t^{2/3} \leq s_{i+1} - s_i \leq 2 \kappa t^{2/3}$ for $i = 0, 1, \dots, K-1$.  Note that $K \leq Ct^{1/3}$.

By Lemma \ref{EZbound}, for $i = 0, 1, \dots, K-1$,
\begin{equation}\label{EZsi}
\E[Z(s_{i+1})|{\cal F}_{s_i}] =\exp \bigg( - (3 \pi^2)^{1/3} \big((t-s_i)^{1/3} - (t-s_{i+1})^{1/3} \big) \bigg) \bigg( \frac{t-s_{i+1}}{t-s_i} \bigg)^{1/6}Z(s_i) D_i,
\end{equation}
where
\begin{equation}\label{Dibound}
\exp(-C_3 \delta^{-1/3} t^{-1/3}) \leq D_i \leq \exp(C_4 \delta^{-1/3} t^{-1/3}).
\end{equation}
Because the particles alive at time $s_{i+1}$ are a subset of the particles that would be alive at time $s_{i+1}$ if particles were killed at $L(s_i)$, rather than $L(s)$, for $s \in [s_i, s_{i+1}]$, and the right-hand side of (\ref{Dineq}) is bounded by a constant when $s \geq \kappa t^{2/3}$ and $L \leq ct^{1/3}$, it follows from (\ref{Yexp}) that
\begin{equation}\label{EYsi}
\E[Y(s_{i+1})|{\cal F}_{s_i}] \leq C Z(s_i)
\end{equation}
for $i = 0, 1, \dots, K-1$.  Let $$Z'(s_{i+1}) = \sum_{i=1}^{N(s_{i+1})} e^{\sqrt{2} X_i(s_{i+1})} \sin \bigg( \frac{\pi X_i(s_{i+1})}{L(s_i)} \bigg),$$ which is the same as $Z(s_{i+1})$ except $L(s_i)$ rather than $L(s_{i+1})$ appears in the denominator.  Because $\sin(\pi x/L(s_{i+1})) \leq C \sin (\pi x/L(s_i))$ for all $x \in [0, L(s_{i+1})]$, we have $Z(s_{i+1}) \leq CZ'(s_{i+1})$.  By Lemma \ref{varZlem}, if there is a single particle at $x$ at time $s_i$, then
$$\mbox{Var}(Z(s_{i+1})|{\cal F}_{s_i}) \le \E[Z(s_{i+1})^2|{\cal F}_{s_i}] \leq C\E[Z'(s_{i+1})^2|{\cal F}_{s_i}] \leq \frac{Ce^{\sqrt{2} x}e^{\sqrt{2}L(s_i)}(s_{i+1} - s_i)}{L(s_i)^4}.$$  Because particles move and branch independently, it follows by summing over the particles at time $s_i$ that
\begin{equation}\label{varZsi}
\mbox{Var}(Z(s_{i+1})|{\cal F}_{s_i}) \leq \frac{CY(s_i) e^{\sqrt{2} L(s_i)} (s_{i+1} - s_i)}{L(s_i)^4} \leq C t^{-2/3} Y(s_i)e^{\sqrt{2} L(s_i)}.
\end{equation}

Using the conditional variance formula, equations (\ref{EZsi}), (\ref{Dibound}), and (\ref{varZsi}), and the fact that $s<(1-\delta)t$,
\begin{align}
\mbox{Var}&(Z(s_{i+1})) = \E[\mbox{Var}(Z(s_{i+1})|{\cal F}_{s_i})] + \mbox{Var}(\E[Z(s_{i+1})|{\cal F}_{s_i}]) \nonumber \\
&\leq C t^{-2/3} e^{\sqrt{2} L(s_i)} \E[Y(s_i)] + e^{ - 2(3 \pi^2)^{1/3} \big((t-s_i)^{1/3} - (t-s_{i+1})^{1/3} \big) } \bigg( \frac{t-s_{i+1}}{t-s_i} \bigg)^{1/3} \mbox{Var}(D_i Z(s_i)). \nonumber \\
&\leq C t^{-2/3} e^{\sqrt{2} L(s_i)} \E[Y(s_i)] + e^{2C_4 \delta^{-1/3} t^{-1/3}} e^{ - 2(3 \pi^2)^{1/3} \big((t-s_i)^{1/3} - (t-s_{i+1})^{1/3} \big) }\mbox{Var}(Z(s_i)). \nonumber
\end{align}
Therefore, by induction,
\begin{align}
\mbox{Var}(Z(s)) &\leq C t^{-2/3} \sum_{i=0}^{K-1} e^{\sqrt{2} L(s_i)} \bigg( \prod_{j=i+1}^{K-1} e^{2C_4 \delta^{-1/3} t^{1/3}} e^{ - 2(3 \pi^2)^{1/3} \big((t-s_j)^{1/3} - (t-s_{j+1})^{1/3} \big) } \bigg) \E[Y(s_i)] \nonumber \\
&\leq C t^{-2/3} \sum_{i=0}^{K-1} e^{\sqrt{2} L(s_i)} e^{2K C_4 \delta^{-1/3} t^{-1/3}} e^{ - 2(3 \pi^2)^{1/3} \big((t-s_{i+1})^{1/3} - (t-s)^{1/3} \big) } \E[Y(s_i)]. \nonumber
\end{align}
By (\ref{EYsi}), for $i = 1, \dots, K-1$, we have $\E[Y(s_i)] = \E[\E[Y(s_i)|{\cal F}_{s_{i-1}}]] \leq C \E[Z(s_{i-1})]$.
Because $K \leq Ct^{1/3}$, we have $e^{2KC_4 \delta^{-1/3} t^{-1/3}} \leq C$.  Therefore,
\begin{align}\label{mainvareq}
\mbox{Var}(Z(s)) &\leq C t^{-2/3} \sum_{i=1}^{K-1} e^{\sqrt{2} L(s_i)} e^{ - 2(3 \pi^2)^{1/3} \big((t-s_{i+1})^{1/3} - (t-s)^{1/3} \big) }  \E[Z(s_{i-1})] \nonumber \\
&\qquad+ C t^{-2/3} e^{\sqrt{2} L(0)} e^{ - 2(3 \pi^2)^{1/3} \big((t-s_{1})^{1/3} - (t-s)^{1/3} \big) }  Y(0).
\end{align}
Denote the two terms on the right-hand side of (\ref{mainvareq}) by $T_1$ and $T_2$.

Because $[(t-s)/t]^{1/6}$ is bounded above and below by positive constants when $0 \leq s \leq (1 - \delta)t$, it follows from Lemma \ref{EZbound} that there are constants $C'$ and $C''$, depending on $\delta$, such that for $i = 0, 1, \dots, K$, $$C' Z(0) \exp \big( -(3 \pi^2)^{1/3} (t^{1/3} - (t - s_i)^{1/3}) \big) \leq \E[Z(s_i)] \leq C'' Z(0) \exp \big( -(3 \pi^2)^{1/3} (t^{1/3} - (t - s_i)^{1/3}) \big).$$
Therefore, using that $\sqrt{2} c = (3 \pi^2)^{1/3}$,
\begin{align}\label{prelimT1}
T_1 &\leq C t^{-2/3} \sum_{i=1}^{K-1} \exp \bigg(\sqrt{2} L(s_i) - 2(3 \pi^2)^{1/3} \big((t-s_{i+1})^{1/3} - (t-s)^{1/3} \big) \nonumber \\
&\qquad - (3 \pi^2)^{1/3}(t^{1/3} - (t - s_{i-1})^{1/3}) \bigg) Z(0) \nonumber \\
&= C t^{-2/3} \sum_{i=1}^{K-1} \exp \big( (3 \pi^2)^{1/3} (t - s_i)^{1/3} - 2(3 \pi^2)^{1/3}((t - s_{i+1})^{1/3} - (t - s)^{1/3}) \nonumber \\
&\qquad - (3 \pi^2)^{1/3}(t^{1/3} - (t - s_{i-1})^{1/3}) \big) Z(0) \nonumber \\
&= C t^{-2/3} \exp \big( 2(3 \pi^2)^{1/3} (t - s)^{1/3} - (3 \pi^2)^{1/3} t^{1/3} \big) Z(0) \nonumber \\
&\qquad\times \sum_{i=1}^{K-1} \exp \big( (3 \pi^2)^{1/3} ((t - s_i)^{1/3} - 2(t - s_{i+1})^{1/3} + (t - s_{i-1})^{1/3}) \big).
\end{align}
For $i = 0, 1, \dots, K-1$, we have $t-s_{i+1} \geq \delta t$, and so $(t - s_{i})^{1/3} - (t - s_{i+1})^{1/3} \leq C$.  Therefore, the sum on the right-hand side of (\ref{prelimT1}) is bounded by $C(K-1) \leq C t^{1/3}$.  Thus, using  $t>1$ and Lemma \ref{EZbound} again,
\begin{align}\label{T1}
T_1 &\leq C t^{-1/3} \exp \big( (3 \pi^2)^{1/3} t^{1/3} \big) \exp \big( 2(3 \pi^2)^{1/3}((t - s)^{1/3} - t^{1/3}) \big) \frac{Z(0)^2}{Z(0)} \nonumber \\
&\leq C t^{-1/3} \exp \big( (3 \pi^2)^{1/3} t^{1/3} \big) \frac{\E[Z(s)]^2}{Z(0)} \nonumber \\
&\leq \frac{C e^{\sqrt{2} L(0)} \E[Z(s)]^2}{L(0) Z(0)}.
\end{align}
Also, using that $t^{1/3} - (t - s_1)^{1/3} \leq C$,
\begin{align}\label{T2}
T_2 &= C t^{-2/3} e^{\sqrt{2} L(0)} \exp \big(- 2(3 \pi^2)^{1/3}((t - s_1)^{1/3} - (t - s)^{1/3}) \big) Y(0) \nonumber \\
&\leq C t^{-2/3} e^{\sqrt{2} L(0)} \exp \big( - 2(3 \pi^2)^{1/3}(t^{1/3} - (t - s)^{1/3}) \big) Y(0) \nonumber \\
&\leq \frac{C e^{\sqrt{2} L(0)} Y(0) \E[Z(s)]^2}{L(0)^2 Z(0)^2}.
\end{align}
The result now follows from (\ref{mainvareq}), (\ref{T1}), and (\ref{T2}).
\end{proof}

\section{Number and configuration of particles}\label{pfsec}

In this section, we return to the model presented in the introduction, in which there is initially a single particle at $x$ and we are concerned with the asymptotic behavior of the process as $x \rightarrow \infty$.

\subsection{The process before time $\kappa x^2$}

We first consider how the branching Brownian motion evolves during the initial period between time $0$ and time $\kappa x^2$, where $\kappa > 0$ is an arbitrary positive constant.  We will use the following result of Neveu \cite{nev87}.

\begin{Lemma}\label{nevlem}
Consider branching Brownian motion with drift $-\sqrt{2}$ and no absorption, started with a single particle at the origin.  For each $y \geq 0$, let $K(y)$ be the number of particles that reach $-y$ in a modified process in which particles are killed upon reaching $-y$.  Then there exists a random variable $W$, with $P(0 < W < \infty) = 1$ and $E[W] = \infty$, such that $$\lim_{y \rightarrow \infty} y e^{-\sqrt{2} y} K(y) = W \hspace{.1in}a.s.$$
\end{Lemma}

For our process which begins with a single particle at $x$, let $K(y)$ be the number of particles that would reach $x - y$, if particles were killed upon reaching $x - y$.  Note that $K(y) < \infty$ almost surely because of Kesten's result \cite{kesten} that critical branching Brownian motion with absorption dies out.  If $y$ is sufficiently large, then $y e^{-\sqrt{2} y} K(y)$ will have approximately the same distribution as the random variable $W$ in Lemma \ref{nevlem}.  Our strategy for studying the branching Brownian motion between time $0$ and time $\kappa x^2$ will be to choose a sufficiently large constant $y$, wait for $K(y)$ particles to reach $x - y$, and then consider $K(y)$ independent branching Brownian motions started from $x - y$.

Let $\alpha \in \R$, and
let
\begin{equation}\label{Zadef}
Z_{\alpha} = \sum_{i=1}^{N(\kappa x^2)} e^{\sqrt{2} X_i(\kappa x^2)} \sin \bigg( \frac{\pi X_i(\kappa x^2)}{x + \alpha} \bigg) {\bf 1}_{\{X_i(\kappa x^2) \leq x + \alpha\}}.
\end{equation}
The following result describes the behavior of the configuration of particles at time $\kappa x^2$.

\begin{Lemma}\label{smalltime}
For all $\varepsilon > 0$, there exists a positive constant $C_5$, depending on $\kappa$ and $\varepsilon$ but not on $x$, such that for sufficiently large $x$,
\begin{equation}\label{Yalpha}
\P\big(Y(\kappa x^2) \leq C_5 x^{-1} e^{\sqrt{2} x} \big) \geq 1 - \varepsilon.
\end{equation}
Also, there exist positive constants $C_6$ and $C_7$, depending on $\kappa$ and $\varepsilon$ but not on $x$ or $\alpha$, such that for sufficiently large $x$,
\begin{equation}\label{Zalpha}
\P\big(C_6 x^{-1} e^{\sqrt{2} x} \leq Z_{\alpha} \leq C_7 x^{-1} e^{\sqrt{2} x}\big) \geq 1 - \varepsilon.
\end{equation}
Furthermore,
\begin{equation}\label{rightmost}
\lim_{x \rightarrow \infty} \P\big( X_1(\kappa x^2) \leq x + \alpha \big) = 1.
\end{equation}
\end{Lemma}

\begin{proof} Choose $\eta > 0$ sufficiently small and $B > 0$ sufficiently large such that the random variable $W$ in Lemma \ref{nevlem} satisfies $\P(W \leq 2 \eta) < \varepsilon/8$ and $\P(W \geq B - \eta) < \varepsilon/8$.  By Lemma \ref{nevlem}, we can choose $y > 0$ large enough that, for some random variable $W$ having the same distribution as the random variable $W$ in Lemma \ref{nevlem}, $$\P(|y e^{-\sqrt{2} y} K(y) - W| \geq \eta) < \frac{\varepsilon}{8}.$$  These conditions imply that
\begin{equation}\label{delbound}
\P(ye^{-\sqrt{2} y} K(y) \leq \eta) < \frac{\varepsilon}{4}
\end{equation}
and
\begin{equation}\label{Bbound}
\P(ye^{-\sqrt{2} y} K(y) \geq B) < \frac{\varepsilon}{4}.
\end{equation}
We can also choose $y$ to be large enough that $y \geq 2|\alpha|$ and $B e^{-\sqrt{2} \alpha}/y < \varepsilon/8$.

For $1 \leq i \leq N(\kappa x^2)$ and $0 \leq s \leq \kappa x^2$, let $x_i(s)$ be the position of the particle at time $s$ that is the ancestor of the particle at the location $X_i(\kappa x^2)$ at time $\kappa x^2$.  Let $v_i = \inf\{s: x_i(s) = x - y\}$.   Let $0 < u_1 < \dots < u_{K(y)}$ denote the times at which particles would hit $x - y$, if particles were killed upon reaching $x - y$.  Note that $\{v_1, \dots, v_{N(\kappa x^2)}\} \subset \{u_1, \dots, u_{K(y)}\}$.
Let ${\cal G}$ denote the $\sigma$-field generated by the set of times $\{u_1, \dots, u_{K(y)}\}$.  We can choose a positive number $\rho > 0$, depending on $y$ but not on $x$, such that
\begin{equation}\label{rhoprob}
\P(u_{K(y)} \leq \rho) > 1 - \frac{\varepsilon}{8}.
\end{equation}
Throughout the proof, we will assume that $x$ is large enough that $x \geq y$, so that particles are not killed at the origin before reaching $x - y$, and that $\kappa x^2/2 \geq \rho$, so that with high probability all particles will have reached $x - y$ well before time $\kappa x^2$.

Let
\begin{equation}\label{Mdef}
M(s) = \sum_{i=1}^{N(s)} X_i(s) e^{\sqrt{2} X_i(s)}.
\end{equation}
It is well-known (see, for example, Lemma 2 of \cite{hh07}) that the process $(M(s), s \geq 0)$ is a martingale.
If there is initially a single particle at $x - y$, then by the Optional Sampling Theorem,
the probability that some particle eventually reaches $x + \alpha$ is at most $$\frac{(x-y) e^{\sqrt{2}(x-y)}}{(x + \alpha) e^{\sqrt{2} (x + \alpha)}}.$$  Therefore, conditional on ${\cal G}$, the probability that some descendant of a particle that reaches $x-y$ eventually reaches $x + \alpha$ is at most $$\frac{K(y) (x - y) e^{\sqrt{2}(x-y)}}{(x + \alpha) e^{\sqrt{2} (x + \alpha)}} \leq \frac{e^{-\sqrt{2} \alpha}}{y} \cdot y e^{-\sqrt{2} y} K(y).$$  Thus, the unconditional probability that some descendant of a particle that reaches $x - y$ eventually reaches $x + \alpha$ is at most $$\P(y e^{-\sqrt{2} y} K(y) > B) + \frac{B e^{-\sqrt{2} \alpha}}{y} < \frac{\varepsilon}{4} + \frac{\varepsilon}{8} = \frac{3 \varepsilon}{8}.$$  In particular, $\P(X_1(\kappa x^2) > x + \alpha) \leq \P(u_{K(y)} > \rho) + 3 \varepsilon/8 \leq \varepsilon/2$ for sufficiently large $x$, which by letting $\varepsilon \rightarrow 0$ implies (\ref{rightmost}).

Let $S(\alpha) = \{i: x_i(s) < x + \alpha \mbox{ for all }s \in [v_i, \kappa x^2]\}$.  Then let
$$Y'_{\alpha} = \sum_{i=1}^{N(\kappa x^2)} e^{\sqrt{2} X_i(\kappa x^2)} {\bf 1}_{\{i \in S(\alpha)\}}$$ and
$$Z'_{\alpha} = \sum_{i=1}^{N(\kappa x^2)} e^{\sqrt{2} X_i(\kappa x^2)} \sin \bigg( \frac{\pi X_i(\kappa x^2)}{x + \alpha} \bigg) {\bf 1}_{\{i \in S(\alpha)\}}.$$
The argument in the previous paragraph implies that
\begin{equation}\label{YZsame}
\P\big(Y_{\alpha}' = Y(\kappa x^2) \mbox{ and }Z_{\alpha}' = Z_{\alpha} \big) \geq 1 - \frac{\varepsilon}{2}.
\end{equation}
By the Strong Markov Property, the configuration of particles at time $\kappa x^2$ has the same distribution as the configuration that we would get by starting with $K(y)$ particles at $x-y$ and stopping their descendants at the times $\kappa x^2 - u_i$.  Furthermore, restricting to particles in $S(\alpha)$ is equivalent to killing particles when they reach $x + \alpha$.  Therefore, the tools of Section \ref{stripsec}, with $L = x + \alpha$, can be used to estimate the first and second moments of $Y_{\alpha}'$ and $Z_{\alpha}'$.

We first apply (\ref{Yexp}) with $s = \kappa x^2 - u_i$, which when $u_{K(y)} \leq \rho$ is at least $\kappa x^2/2$.  Because the right-hand side of (\ref{Dineq}) is bounded by a constant when $s$ is of the order $L^2$, it follows from (\ref{Yexp}) that there is a constant $C$, depending on $\kappa$, such that on the event $\{u_{K(y)} \leq \rho\}$,
$$\E[Y_{\alpha}'|{\cal G}] \leq C K(y) e^{\sqrt{2}(x - y)} \sin \bigg( \frac{\pi (x - y)}{x + \alpha} \bigg).$$  Using $\sim$ to denote that the ratio of the two sides tends to one as $x \rightarrow \infty$, we have
\begin{equation}\label{sinasym}
\sin \bigg( \frac{\pi (x - y)}{x + \alpha} \bigg) \sim \frac{\pi(y + \alpha)}{x + \alpha} \sim \frac{\pi (y + \alpha)}{x}.
\end{equation}
Because $y \geq 2|\alpha|$, it follows that there exists a constant $C_8$ such that on the event $\{u_{K(y)} \leq \rho\}$, for sufficiently large $x$, $$\E[Y_{\alpha}'|{\cal G}] \leq C_8 x^{-1} e^{\sqrt{2} x} \cdot y e^{-\sqrt{2} y} K(y).$$  Therefore, choosing $C_5 = 8 C_8 B/\varepsilon$ and using (\ref{Bbound}), (\ref{rhoprob}), and the conditional Markov's inequality,
\begin{align}\label{Yprime}
\P(Y_{\alpha}' \geq C_5 x^{-1} e^{\sqrt{2} x}) &\leq \P(u_{K(y)} > \rho) + \P \bigg( y e^{-\sqrt{2} y} K(y) \geq B \bigg) + \P \bigg(Y_{\alpha}' \geq \frac{8 \E[Y_{\alpha}'|{\cal G}]}{\varepsilon} \bigg) \nonumber \\
&\leq \frac{\varepsilon}{8} + \frac{\varepsilon}{4} + \frac{\varepsilon}{8} = \frac{\varepsilon}{2}.
\end{align}
The result (\ref{Yalpha}) now follows from (\ref{Yprime}) and (\ref{YZsame}).

By (\ref{Zexp}), on the event $\{u_{K(y)} \leq \rho\}$, we have
\begin{align}
&e^{-\pi^2 \kappa x^2/2(x + \alpha)^2} K(y) e^{\sqrt{2} (x - y)} \sin \bigg( \frac{\pi (x - y)}{x + \alpha} \bigg) \nonumber \\
&\qquad \leq \E[Z_{\alpha}|{\cal G}] \leq e^{-\pi^2 (\kappa x^2 - \rho)/2(x + \alpha)^2} K(y) e^{\sqrt{2} (x - y)} \sin \bigg( \frac{\pi (x - y)}{x + \alpha} \bigg). \nonumber
\end{align}
Because (\ref{sinasym}) holds and $e^{-\pi^2 \kappa x^2/2(x+\alpha)^2} \sim e^{-\pi^2 \kappa/2} \sim e^{-\pi^2 (\kappa x^2 - \rho)/2(x + \alpha)^2}$, there are constants $C_9$ and $C_{10}$, depending on $\kappa$, such that
\begin{equation}\label{EZG}
C_9 x^{-1} e^{\sqrt{2} x} \cdot y e^{-\sqrt{2} y} K(y) \leq \E[Z_{\alpha}'|{\cal G}] \leq C_{10} x^{-1} e^{\sqrt{2} x} \cdot y e^{-\sqrt{2} y} K(y)
\end{equation}
when $u_{K(y)} \leq \rho$ for sufficiently large $x$.  Furthermore, by applying Lemma \ref{varZlem} to the configuration with a single particle at $x-y$ at time zero and then summing over the particles, we get $$\mbox{Var}(Z_{\alpha}'|{\cal G}) \leq \frac{C K(y) e^{\sqrt{2} (x - y)} e^{\sqrt{2} (x + \alpha)} \kappa x^2}{2(x + \alpha)^4} \leq \frac{C e^{\sqrt{2} \alpha} \cdot ye^{-\sqrt{2} y} K(y)}{y} \big( x^{-1} e^{\sqrt{2} x} \big)^2$$
for sufficiently large $x$.
By the conditional Chebyshev's Inequality, on the event $\{u_{K(y)} \leq \rho\}$, $$\P \bigg( \big| Z_{\alpha}' - \E[Z_{\alpha}'|{\cal G}] \big| > \frac{1}{2} \E[Z_{\alpha}'|{\cal G}] \bigg| {\cal G} \bigg) \leq \frac{4\mbox{Var}(Z_{\alpha}'|{\cal G})}{(\E[Z_{\alpha}'|{\cal G}])^2} \leq \frac{Ce^{\sqrt{2} \alpha}}{y \cdot y e^{-\sqrt{2} y} K(y)}.$$  In view of (\ref{delbound}), it follows that for $y$ large enough that $Ce^{\sqrt{2} \alpha}/(\eta y) < \varepsilon/8$ and sufficiently large $x$, $$\P \bigg( \big| Z_{\alpha}' - \E[Z_{\alpha}'|{\cal G}] \big| > \frac{1}{2} \E[Z_{\alpha}'|{\cal G}] \bigg) \leq \P(u_{K(y)} > \rho) + \P(y e^{-\sqrt{2} y} K(y) \leq \eta) + \frac{Ce^{\sqrt{2} \alpha}}{\eta y} < \frac{\varepsilon}{2}.$$  Combining this result with (\ref{EZG}), we get that for sufficiently large $x$, the event
$$\frac{C_9}{2} \cdot x^{-1} e^{\sqrt{2} x} \cdot y e^{-\sqrt{2} y} K(y) \leq Z_{\alpha}' \leq \frac{3C_{10}}{2} \cdot x^{-1} e^{\sqrt{2} x} \cdot y e^{-\sqrt{2} y} K(y)$$
holds with probability at least $1 - \varepsilon/2$.  Thus, using (\ref{delbound}) and (\ref{Bbound}), for sufficiently large $x$ we have $$\frac{C_9 \eta}{2} \cdot x^{-1} e^{\sqrt{2} x} \leq Z_{\alpha}' \leq \frac{3BC_{10}}{2} \cdot x^{-1} e^{\sqrt{2} x}$$
with probability at least $1 - \varepsilon$.  The result (\ref{Zalpha}) now follows by setting $C_6 = C_9 \eta/2$ and $C_7 = 3BC_{10}/2$ and invoking (\ref{YZsame}).
\end{proof}

\subsection{A lower bound for $Z(s)$}\label{Zlow}

Let $t = \tau x^3 = 2 \sqrt{2} x^3/(3 \pi^2)$.
For $0 < s < t$, recall that $$L(s) = x \bigg(1 - \frac{3 \pi^2 s}{2 \sqrt{2} x^3} \bigg)^{1/3} = c (t - s)^{1/3}$$ as in (\ref{Ldef}), and let $$Z(s) = \sum_{i=1}^{N(s)} e^{\sqrt{2} X_i(s)} \sin \bigg( \frac{\pi X_i(s)}{L(s)} \bigg) {\bf 1}_{\{X_i(s) \leq L(s)\}}.$$  Our goal in this subsection is to find a lower bound for $Z(s)$.  Such a bound will be provided by Proposition \ref{Zlower} below.

To prove this result, we will consider the following new process, which will also be useful in later subsections.  Fix $\alpha \in \R$, and let $t_{\alpha} = \tau (x + \alpha)^3$, so that $ct_{\alpha}^{1/3} = x + \alpha$, where $c$ is defined in (\ref{cdef}).  For $0 \leq s \leq t_{\alpha}$, let $L_{\alpha}(s) = c(t_{\alpha} - s)^{1/3}$.  Note that $L_0(s) = L(s)$.  Now suppose that, in addition to being killed at the origin, particles to the right of $x + \alpha$ are killed at time $\kappa x^2$, and for $\kappa x^2 < s < t_{\alpha} + \kappa x^2$, particles are killed at time $s$ if they reach $L_{\alpha}(s - \kappa x^2)$.  Let $N_{\alpha}(s)$ be the number of particles alive at time $s$, and let $X_{1,\alpha}(s) \geq \dots \geq X_{N_{\alpha}(s), \alpha}(s)$ denote the positions of these particles at time $s$.  Let $$Z_{\alpha}(s) = \sum_{i=1}^{N_{\alpha}(s)} e^{\sqrt{2} X_{i,\alpha}(s)} \sin \bigg( \frac{\pi X_{i,\alpha}(s)}{L_{\alpha}(s - \kappa x^2)} \bigg).$$  Note that $Z_{\alpha}(\kappa x^2)$ is the same as $Z_{\alpha}$ defined in (\ref{Zadef}).  Also, let
\begin{equation}\label{Yalphadef}
Y_{\alpha}(s) = \sum_{i=1}^{N_{\alpha}(s)} e^{\sqrt{2} X_{i,\alpha}(s)}.
\end{equation}

\begin{Prop}\label{Zlower}
For all $\varepsilon > 0$, there exists a constant $C > 0$ depending on $\kappa$, $\delta$, and $\varepsilon$ such that for sufficiently large $x$,
$$\P\bigg(Z(s) \geq C x^{-1} \exp \big( (3 \pi^2)^{1/3}(t - s)^{1/3} \big) \bigg) > 1 - \varepsilon$$ for all $s \in [2 \kappa x^2, (1 - \delta)t]$.
\end{Prop}

\begin{proof}
We consider the process defined above.  Recall that $({\cal F}_u)_{u \geq 0}$ is the natural filtration associated with the branching Brownian motion.  By Lemma \ref{EZbound} and the Markov property, there exist positive constants $C'$ and $C''$, depending on $\kappa$ and $\delta$, such that for all $s \in [2 \kappa x^2, (1 - \delta/2)t_{\alpha}]$,
$$C' Z_{\alpha} G_0(s - \kappa x^2) \leq \E[Z_{\alpha}(s)|{\cal F}_{\kappa x^2}] \leq C'' Z_{\alpha} G_0(s - \kappa x^2).$$
Because $(t_{\alpha} - (s - \kappa x^2))^{1/3} - (t_{\alpha}-s)^{1/3}$ is bounded by a constant, it follows from (\ref{Gdef}) that
\begin{align}\label{EZalphas}
C' Z_{\alpha} \exp \big( -(3 \pi^2)^{1/3} \big(t_{\alpha}^{1/3} - (t_{\alpha}-s)^{1/3} \big) \big) &\leq \E[Z_{\alpha}(s)|{\cal F}_{\kappa x^2}] \nonumber \\
&\leq C'' Z_{\alpha} \exp \big( -(3 \pi^2)^{1/3} \big(t_{\alpha}^{1/3} - (t_{\alpha}-s)^{1/3} \big) \big).
\end{align}
Likewise, by Proposition \ref{VarZProp},
\begin{align}
\mbox{Var}(Z_{\alpha}(s)|{\cal F}_{\kappa x^2}) &\leq C \E[Z_{\alpha}(s)|{\cal F}_{\kappa x^2}]^2 \bigg( \frac{e^{\sqrt{2} L_{\alpha}(0)}}{L_{\alpha}(0) Z_{\alpha}} + \frac{e^{\sqrt{2} L_{\alpha}(0)} Y(\kappa x^2)}{L_{\alpha}(0)^2 Z_{\alpha}^2} \bigg) \nonumber \\
&= C \E[Z_{\alpha}(s)|{\cal F}_{\kappa x^2}]^2 \bigg( \frac{e^{\sqrt{2} x} e^{\sqrt{2} \alpha}}{(x + \alpha) Z_{\alpha}} + \frac{e^{\sqrt{2} x} e^{\sqrt{2} \alpha} Y(\kappa x^2)}{(x + \alpha)^2 Z_{\alpha}^2} \bigg). \nonumber
\end{align}
Let $A$ be the event that $Y(\kappa x^2) \leq C_5 x^{-1} e^{\sqrt{2} x}$ and $Z_{\alpha} \geq C_6 x^{-1} e^{\sqrt{2} x}$, where $C_5$ and $C_6$ are the constants from Lemma \ref{smalltime} applied with $\varepsilon/8$ in place of $\varepsilon$.  Lemma \ref{smalltime} then gives $\P(A) > 1 - \varepsilon/4$ for sufficiently large $x$.  On $A$, we have
$$\mbox{Var}(Z_{\alpha}(s)|{\cal F}_{\kappa x^2}) \leq C \E[Z_{\alpha}(s)|{\cal F}_{\kappa x^2}]^2 e^{\sqrt{2} \alpha} \bigg( \frac{x}{C_6(x + \alpha)} + \frac{C_5 x}{C_6^2 (x + \alpha)^2} \bigg).$$  Therefore, if $\alpha$ is chosen to be a large enough negative number that $C e^{\sqrt{2} \alpha}/C_6 < \varepsilon/8$, then $\mbox{Var}(Z_{\alpha}(s)|{\cal F}_{\kappa x^2}) \leq (\varepsilon/8) \E[Z_{\alpha}(s)|{\cal F}_{\kappa x^2}]^2$ on $A$ for sufficiently large $x$.  It follows from the conditional Chebyshev's Inequality that for sufficiently large $x$,
\begin{equation}\label{PZalpha}
\P\bigg(Z_{\alpha}(s) < \frac{1}{2} \E[Z_{\alpha}(s)|{\cal F}_{\kappa x^2}] \bigg) \leq \P(A^c) + \frac{4 \varepsilon}{8} < \frac{3\varepsilon}{4}.
\end{equation}
By (\ref{EZalphas}), on $A$ we have
$$\E[Z_{\alpha}(s)|{\cal F}_{\kappa x^2}] \geq C x^{-1} \exp \big( \sqrt{2} x - (3 \pi^2)^{1/3}  \big(t_{\alpha}^{1/3} - (t_{\alpha}-s)^{1/3} \big) \big).$$  Thus, using (\ref{PZalpha}) and the fact that $\P(A^c) < \varepsilon/4$, there is a positive constant $C$ such that for all $s \in [2 \kappa x^2, (1 - \delta/2) t_{\alpha}]$, $$\P \bigg( Z_{\alpha}(s) \geq C x^{-1}  \exp \big( \sqrt{2} x - (3 \pi^2)^{1/3}  \big(t_{\alpha}^{1/3} - (t_{\alpha}-s)^{1/3} \big) \big) \bigg) \geq 1 - \varepsilon$$
for sufficiently large $x$.  Note that $|t_{\alpha}^{1/3} - t^{1/3}|$ is bounded by a constant which depends on $\alpha$, and thus on $\varepsilon$.  Likewise,
\begin{equation}\label{talphabound}
\sup_{\kappa x^2 \leq s \leq (1 - \delta/2)t_{\alpha}} |(t_{\alpha} - s)^{1/3} - (t - s)^{1/3}|
\end{equation}
is bounded by a constant which depends on $\alpha$ and $\delta$.  Furthermore, we have $\sqrt{2} x = (3 \pi^2)^{1/3} t^{1/3}$.  Because $(1 - \delta/2)t_{\alpha} \geq (1 - \delta) t$ for sufficiently large $x$, we obtain the result of the proposition with $Z_{\alpha}(s)$ in place of $Z(s)$, provided that $\alpha$ is a sufficiently large negative number.

To complete the proof, recall that $L(s) = c(t - s)^{1/3}$ and $L_{\alpha}(s - \kappa x^2) = c(t_{\alpha} - s + \kappa x^2)^{1/3}$, where $t = \tau x^3$ and $t_{\alpha} = \tau (x + \alpha)^3$.
Therefore, there is a constant $\alpha_0 < 0$ such that if $\alpha < \alpha_0$, then $L_{\alpha}(s - \kappa x^2) < L(s)$ for sufficiently large $x$.  Also, $L(s)/2 < L_{\alpha}(s - \kappa x^2)$ for sufficiently large $x$.  Thus, if $\alpha < \alpha_0$, there exists a constant $C$ such that for sufficiently large $x$, $$\sin \bigg( \frac{\pi z}{L_{\alpha}(s - \kappa x^2)} \bigg) \leq C \sin \bigg( \frac{\pi z}{L(s)} \bigg)$$
for all $z \in [0, L_{\alpha}(s - \kappa x^2)]$.  Because killing particles at a right boundary can only reduce the number of particles in the system, it follows that if $\alpha < \alpha_0$, then $Z_{\alpha}(s) \leq C Z(s)$ for sufficiently large $x$.  The result follows.
\end{proof}

\subsection{Upper bounds for $Z(s)$ and $Y(s)$}

Recall that $t = \tau x^3$.  The next lemma shows that it is unlikely for any particle ever to get far to the right of $L(s)$ for $s \in [2 \kappa x^2, (1 - \delta)t]$.

\begin{Lemma}\label{killalpha}
Let $\varepsilon > 0$.  For all $\alpha > 0$, let $t_{\alpha} = \tau(x + \alpha)^3$, and let $L_{\alpha}(s) = c(t_{\alpha} - s)^{1/3}$ for $0 \leq s \leq t_{\alpha}$.  Then there exists a positive constant $C_{11}$, depending on $\kappa$, $\delta$, and $\varepsilon$ but not on $\alpha$ or $x$, such that for sufficiently large $x$, $$\P\big(X_1(s) \leq L_{\alpha}(s - \kappa x^2) \mbox{ for all }s \in [\kappa x^2, (1 - \delta)t] \big) \geq 1 - \varepsilon - C_{11} e^{-\sqrt{2} \alpha}.$$
\end{Lemma}

\begin{proof}
Suppose there is a particle at the location $z \leq c t_{\alpha}^{1/3}=x+\alpha$ at time $\kappa x^2$.
By Lemma \ref{L:sumRj} with $t = t_{\alpha}$, the probability that a descendant of this particle reaches $L_{\alpha}(s - \kappa x^2)$ for some $s \in [\kappa x^2, (1 - \delta/2)t_{\alpha}]$ is at most
$$C e^{-(3 \pi^2 t_{\alpha})^{1/3}} \bigg( e^{\sqrt{2} z} \sin \bigg( \frac{\pi z}{L_{\alpha}(0)} \bigg) t_{\alpha}^{1/3} + z e^{\sqrt{2} z} t_{\alpha}^{-1/3} \bigg).$$  Therefore, using the bound $z t_{\alpha}^{-1/3} \leq c$ and applying the Markov property, we get that the conditional probability, given ${\cal F}_{\kappa x^2}$, on the event $X_1(\kappa x^2)<x+\alpha,$  that a particle reaches $L_{\alpha}(s - \kappa x^2)$ for $\kappa x^2 \leq s \leq (1 - \delta/2)t_{\alpha}$ is at most
\begin{equation}\label{hitbd}
Ce^{-\sqrt{2}(x + \alpha)}  \left( t_{\alpha}^{1/3} Z_{\alpha}(\kappa x^2)  + Y(\kappa x^2) \right).
\end{equation}
Let $A$ be the event that  $X_1(\kappa x^2)<x+\alpha,$ $Y(\kappa x^2) \leq C_5 x^{-1} e^{\sqrt{2} x}$ and $Z_{\alpha}(\kappa x^2) \leq C_7 x^{-1} e^{\sqrt{2} x}$, where $C_5$ and $C_7$ are the constants from Lemma \ref{smalltime} with $\varepsilon/3$ in place of $\varepsilon$.  On $A$, for sufficiently large $x$,
the expression in (\ref{hitbd}) is at most $$C t_{\alpha}^{1/3} x^{-1} e^{-\sqrt{2} \alpha} + C x^{-1} e^{-\sqrt{2} \alpha} \leq C_{11} e^{-\sqrt{2} \alpha}.$$
Because $\P(A) > 1 - \varepsilon$ for sufficiently large $x$ by Lemma \ref{smalltime} and the fact that $(1 - \delta/2)t_{\alpha} \geq (1 - \delta)t$ for sufficiently large $x$, the result follows.
\end{proof}

The next lemma shows that at any fixed time $s \in [2 \kappa x^2, (1 - \delta)t]$, it is unlikely that there is any particle near or to the right of $L(s)$.

\begin{Lemma}\label{withina}
Let $a > 0$ be a positive constant.  Let $\varepsilon > 0$.  Then for sufficiently large $x$, we have
$$\P(X_1(s) > L(s) - a) < \varepsilon$$ for all $s \in [2 \kappa x^2, (1 - \delta)t]$.
\end{Lemma}

\begin{proof}
We consider the process defined at the beginning of Section \ref{Zlow} in which at time $\kappa x^2$, particles to the right of $x + \alpha$ are killed, and for $\kappa x^2 < s < t_{\alpha} + \kappa x^2$, particles are killed at time $s$ if they reach $L_{\alpha}(s - \kappa x^2)$.  By (\ref{rightmost}), for sufficiently large $x$, the probability that some particle is killed at time $\kappa x^2$ is at most $\varepsilon/4$.  By applying Lemma \ref{killalpha} with $\varepsilon/4$ in place of $\varepsilon$ and choosing $\alpha > 0$ large enough that $C_{11} e^{-\sqrt{2} \alpha} < \varepsilon/4$, we get that the probability that a particle is killed between times $\kappa x^2$ and $(1 - \delta)t$ is at most $\varepsilon/2$.  Thus, with probability at least $1 - 3 \varepsilon/4$, no particle is killed until at least time $(1 - \delta)t$.

Suppose $s \in [2 \kappa x^2, (1 - \delta)t]$.  Let $K_{\alpha}(s)$ be the number of particles at time $s$ between $L(s) - a$ and $L_{\alpha}(s - \kappa x^2)$.
By Lemma \ref{densityprop} with $t_{\alpha}$ in place of $t$, we have
\begin{align}
\E[K_{\alpha}(s)|{\cal F}_{\kappa x^2}] &\leq C t_{\alpha}^{-1/3} e^{-(3 \pi^2)^{1/3}(t_{\alpha}^{1/3} - (t_{\alpha} - s + \kappa x^2)^{1/3})} Z_{\alpha} \nonumber \\
&\qquad \times \int_{L(s) - a}^{L_{\alpha}(s - \kappa x^2)} e^{-\sqrt{2} y} \sin \bigg( \frac{\pi y}{L_{\alpha}(s - \kappa x^2)} \bigg) \: dy. \nonumber
\end{align}
For sufficiently large $x$, the expression $L_{\alpha}(s - \kappa x^2) - (L(s) - a) = c(t_{\alpha} - s + \kappa x^2)^{1/3} - c(t - s)^{1/3} + a$ is bounded above by a constant depending on $\alpha$ and $a$, and thus
$$\int_{L(s) - a}^{L_{\alpha}(s - \kappa x^2)} e^{-\sqrt{2} y} \sin \bigg( \frac{\pi y}{L_{\alpha}(s - \kappa x^2)} \bigg) \: dy
\leq \frac{C e^{-\sqrt{2} L_{\alpha}(s - \kappa x^2)}}{L_{\alpha}(s - \kappa x^2)} \leq C t_{\alpha}^{-1/3} e^{-\sqrt{2} L_{\alpha}(s - \kappa x^2)}.$$
Therefore, on the event that $Z_{\alpha} \leq C_7 x^{-1} e^{\sqrt{2} x}$, where $C_7$ is the constant from Lemma \ref{smalltime} with $\varepsilon/8$ in place of $\varepsilon$, for sufficiently large $x$,
\begin{align}
&\E[K_{\alpha}(s)|{\cal F}_{\kappa x^2}] \nonumber \\
&\qquad \leq C t_{\alpha}^{-2/3} x^{-1} \exp \big( \sqrt{2} x -(3 \pi^2)^{1/3}\big(t_{\alpha}^{1/3} - (t_{\alpha} - s + \kappa x^2)^{1/3} \big) - \sqrt{2} L_{\alpha}(s - \kappa x^2) \big) \nonumber \\
&\qquad \leq C x^{-3} \exp \big( \sqrt{2} x - \sqrt{2}(x + \alpha) + (3 \pi^2)^{1/3} (t_{\alpha} - s + \kappa x^2)^{1/3} - (3 \pi^2)^{1/3}(t_{\alpha} - s + \kappa x^2)^{1/3} \big) \nonumber \\
&\qquad \leq C x^{-3} \nonumber
\end{align}
because the exponential is a constant which depends on $\alpha$.  Therefore, by the conditional Markov's Inequality and Lemma \ref{smalltime}, for sufficiently large $x$, $$\P(K_{\alpha}(s) > 0) \leq \P(Z_{\alpha} > C_7 x^{-1} e^{\sqrt{2} x}) + C x^{-3} < \frac{\varepsilon}{8} + \frac{\varepsilon}{8} = \frac{\varepsilon}{4}.$$
Because with probability at least $1 - 3 \varepsilon/4$, no particle is killed until at least time $(1 - \delta)t$, it follows that for sufficiently large $x$, we have $\P(X_1(s) > L(s) - a) < \varepsilon$ for all $s \in [2 \kappa x^2, (1 - \delta)t]$.
\end{proof}

\begin{Prop}\label{Zupper}
For all $\varepsilon > 0$, there exists a constant $C > 0$ depending on $\kappa$, $\delta$, and $\varepsilon$ such that for sufficiently large $x$,
$$\P\bigg(Z(s) \leq C x^{-1} \exp \big( (3 \pi^2)^{1/3}(t - s)^{1/3} \big) \bigg) > 1 - \varepsilon$$ for all $s \in [2 \kappa x^2, (1 - \delta)t]$.
\end{Prop}

\begin{proof}
We again work with the process defined at the beginning of Section \ref{Zlow}.  By (\ref{EZalphas}) and the conditional Markov's Inequality, there is a constant $C$ depending on $\kappa$, $\delta$, and $\varepsilon$ such that for all $s \in [2 \kappa x^2, (1 - \delta)t]$, $$\P\bigg(Z_{\alpha}(s) \leq C Z_{\alpha} \exp \big( -(3 \pi^2)^{1/3} (t_{\alpha}^{1/3} - (t_{\alpha}-s)^{1/3}) \big) \bigg) > 1 - \frac{\varepsilon}{4}.$$  Therefore, by (\ref{Zalpha}), for all $s \in [2 \kappa x^2, (1 - \delta)t]$, $$\P\bigg(Z_{\alpha}(s) \leq C x^{-1} \exp \big( \sqrt{2} x -(3 \pi^2)^{1/3} (t_{\alpha}^{1/3} - (t_{\alpha}-s)^{1/3}) \big) \bigg) > 1 - \frac{\varepsilon}{2}$$ for sufficiently large $x$.  Because $|t_{\alpha}^{1/3} - t^{1/3}|$ and the expression in (\ref{talphabound}) are bounded by constants depending on $\alpha$ and $\sqrt{2} x = (3 \pi^2t)^{1/3}$, it follows that
\begin{equation}\label{Zalphabd}
\P\bigg(Z_{\alpha}(s) \leq C x^{-1} \exp \big( (3 \pi^2)^{1/3}(t-s)^{1/3} \big) \bigg) > 1 - \frac{\varepsilon}{2}
\end{equation}
for sufficiently large $x$.

From Lemma \ref{killalpha} with $\varepsilon/8$ in place of $\varepsilon$, we see that with probability at least $1 - \varepsilon/8 - C_{11}e^{-\sqrt{2} \alpha}$, no particles are killed between times $\kappa x^2$ and $(1 - \delta)t$.  Therefore, if $\alpha$ is chosen large enough that $C_{11}e^{-\sqrt{2} \alpha} < \varepsilon/8$, then with probability at least $1 - \varepsilon/4$, we have $N_{\alpha}(s) = N(s)$ and $X_i(s) = X_{i,\alpha}(s)$ for $i = 1, \dots, N(s)$.  Furthermore, provided $\alpha$ is also large enough that $L_{\alpha}(s - \kappa x^2) \geq L(s)$, for sufficiently large $x$ it holds that for $0 \leq x \leq L(s)$, we have
$$\sin \bigg( \frac{\pi x}{L_{\alpha}(s - \kappa x^2)} \bigg) \geq C \sin \bigg( \frac{\pi x}{L(s)} \bigg)$$ for some positive constant $C$.  By Lemma \ref{withina}, for sufficiently large $x$ the probability that $X_1(s) > L(s)$ is less than $\varepsilon/4$.  It follows that for sufficiently large $x$, we have $Z_{\alpha}(s) \geq C Z(s)$ with probability at least $1 - \varepsilon/2$.  Combining this observation with (\ref{Zalphabd}) yields the result.
\end{proof}

\begin{Prop}\label{Yupper}
For all $\varepsilon > 0$, there exists a constant $C > 0$ depending on $\kappa$, $\delta$, and $\varepsilon$ such that for sufficiently large $x$,
$$\P\bigg(Y(s) \leq C x^{-1} \exp \big( (3 \pi^2)^{1/3}(t - s)^{1/3} \big) \bigg) > 1 - \varepsilon$$ for all $s \in [2 \kappa x^2, (1 - \delta)t]$.
\end{Prop}

\begin{proof}
We again work with the process defined at the beginning of Section \ref{Zlow}.  Recall the definition of $Y_{\alpha}(s)$ from (\ref{Yalphadef}).  By Lemma \ref{killalpha}, we can choose $\alpha > 0$ sufficiently large that with probability at least $1 - \varepsilon/2$, we have $X_1(s) \leq c(t_{\alpha} - s + \kappa x^2)^{1/3}$ for all $s \in [\kappa x^2, (1 - \delta)t_{\alpha}]$.  Therefore, for all $s \in [2 \kappa x^2, (1 - \delta)t]$, we have $\P(Y_{\alpha}(s) = Y(s)) > 1 - \varepsilon/2$.

By Lemma \ref{densityprop} with $t_{\alpha}$ in place of $t$, for all $s \in [2 \kappa x^2, (1 - \delta)t]$,
\begin{align}
\E[Y_{\alpha}(s)|{\cal F}_{\kappa x^2}] &\leq \frac{C}{L_{\alpha}(s - \kappa x^2)} e^{-(3 \pi^2)^{1/3}(t_{\alpha}^{1/3} - (t_{\alpha} - s + \kappa x^2)^{1/3})} Z_{\alpha} \int_0^{L_{\alpha}(s - \kappa x^2)} \sin \bigg( \frac{\pi y}{L_{\alpha}(s - \kappa x^2)} \bigg) \: dy \nonumber \\
&\leq C e^{-(3 \pi^2)^{1/3}(t_{\alpha}^{1/3} - (t_{\alpha} - s + \kappa x^2)^{1/3})} Z_{\alpha}. \nonumber
\end{align}
By combining this result with the conditional Markov's inequality and (\ref{Zalpha}), we get that there is a constant $C$ such that for sufficiently large $x$, $$\P \bigg( Y_{\alpha}(s) \leq C x^{-1} e^{\sqrt{2} x} e^{-(3 \pi^2)^{1/3}(t_{\alpha}^{1/3} - (t_{\alpha} - s + \kappa x^2)^{1/3})} \bigg) > 1 - \frac{\varepsilon}{2}$$
for all $s \in [2 \kappa x^2, (1 - \delta)t]$.  Because $|(t_{\alpha} - s - \kappa x^2)^{1/3} - (t - s)^{1/3}|$ is bounded by a constant depending on $\alpha$ and $(3 \pi^2)^{1/3} t_{\alpha}^{1/3} = \sqrt{2} (x - \alpha)$, there is a constant $C$ depending on $\alpha$ such that
$$\P\bigg(Y_{\alpha}(s) \leq C x^{-1} \exp \big( (3 \pi^2)^{1/3}(t - s)^{1/3} \big) \bigg) > 1 - \frac{\varepsilon}{2}$$ for all $s \in [2 \kappa x^2, (1 - \delta)t]$.  The result follows because $\P(Y_{\alpha}(s) = Y(s)) > 1 - \varepsilon/2$.
\end{proof}

\subsection{Moments of functions of branching Brownian motion}\label{momfunsec}

Suppose $\kappa > 0$ and $\delta > 0$.  Let $\varepsilon > 0$.  Choose a constant $B > 0$ sufficiently large that if $s = BL^2$, the right-hand side of (\ref{Dineq}) is at most $\varepsilon$.  Now fix a time $s$ such that $$(B + 3 \kappa)x^2 \leq s \leq (1 - \delta)t.$$
Let $f: [0, \infty) \rightarrow \R$ and $\phi: [0,1] \rightarrow \R$ be bounded continuous functions.  Let $\|f\| = \sup_{x \geq 0} |f(x)|$ and $\|\phi\| = \sup_{0 \leq x \leq 1} |\phi(x)|$.  We are interested here in the quantities
\begin{equation}\label{sumfX}
\sum_{i=1}^{N(s)} f(X_i(s)).
\end{equation}
and
\begin{equation}\label{sumphiX}
\sum_{i=1}^{N(s)} e^{\sqrt{2} X_i(s)} \phi\bigg( \frac{X_i(s)}{L(s)} \bigg) {\bf 1}_{\{X_i(s) < L(s)\}}.
\end{equation}

Let $r = s - Bx^2$.  Let $A$ be the event that $X_1(u) \leq L(s)$ for all $u \in [r, s]$.
By Proposition \ref{Yupper}, there is a positive constant $C$ such that
\begin{equation}\label{811}
\P \bigg( Y(r) \leq C x^{-1} \exp((3 \pi^2)^{1/3}(t - r)^{1/3}) \bigg) > 1 - \varepsilon
\end{equation}
for sufficiently large $x$.  Because $L(r) - L(s)$ is bounded above by a constant, Lemma \ref{withina} implies that
\begin{equation}\label{812}
\P(X_1(r) \leq L(s)) > 1 - \varepsilon
\end{equation}
for sufficiently large $x$.  Because $M(r)$, as defined in (\ref{Mdef}), is bounded by $X_1(r)Y(r)$, we have
$$M(r) \leq C L(s) x^{-1} \exp((3 \pi^2)^{1/3}(t - r)^{1/3})$$ when the events in (\ref{811}) and (\ref{812}) both occur.
By the Optional Sampling Theorem, the probability, conditional on ${\cal F}_r$, that some particle reaches $L(s)$ between times $r$ and $s$ is at most $M(r)/(L(s) e^{\sqrt{2} L(s)})$.  Therefore,
\begin{equation}\label{821}
\P(A^c) \leq 2 \varepsilon + C x^{-1} \exp \big( (3 \pi^2)^{1/3}(t - r)^{1/3} - \sqrt{2} L(s) \big).
\end{equation}
Because $\sqrt{2} L(s) = (3 \pi^2)^{1/3} (t - s)^{1/3}$, the exponential on the right-hand side of (\ref{821}) is bounded by a constant.  Therefore, the second term on the right-hand side of (\ref{821}) tends to zero as $x \rightarrow \infty$, and thus $\P(A^c) < 3 \varepsilon$ for sufficiently large $x$.

Let $S$ be the set of all $i \in \{1, \dots, N(s)\}$ such that for all $u \in [r, s]$, the particle at time $u$ that is the ancestor of the particle at $X_i(s)$ at time $s$ is positioned to the left of $L(s)$.  We will work with the quantities $$X(f) = \sum_{i=1}^{N(s)} f(X_i(s)) {\bf 1}_{\{i \in S\}}$$ and $$X'(\phi) = \sum_{i=1}^{N(s)} e^{\sqrt{2} X_i(s)} \phi \bigg( \frac{X_i(s)}{L(s)} \bigg) {\bf 1}_{\{i \in S\}}.$$  Note that $X(f)$ and $X'(\phi)$ equal the sums in (\ref{sumfX}) and (\ref{sumphiX}) respectively on the event $A$, so we have the following result.

\begin{Lemma}\label{fphi}
Suppose $\varepsilon$, $B$, $r$, and $s$ are as defined above.  Then for sufficiently large $x$, with probability greater than $1 - 3 \varepsilon$, the quantity $X(f)$ equals the sum in (\ref{sumfX}) and $X'(\phi)$ equals the sum in (\ref{sumphiX}) for all bounded continuous functions $f: [0, \infty) \rightarrow \R$ and $\phi: [0,1] \rightarrow \R$.
\end{Lemma}

Because $X(f)$ and $X'(\phi)$ are the sums that would be obtained if particles were killed at $L(s)$ between times $r$ and $s$, we can compute conditional moments of $X(f)$ and $X'(\phi)$ by applying Lemma \ref{fXvar} with $Bx^2$ in place of $s$ and $L(s)$ in place of $L$.  For the rest of this subsection, we define $q_u(x,y)$ as in Lemma \ref{stripdensity} with $L(s)$ in place of $L$.

Define
\begin{equation}\label{hatZdef}
{\hat Z} = \sum_{i=1}^{N(r)} e^{\sqrt{2} X_i(r)} \sin \bigg( \frac{\pi X_i(r)}{L(s)} \bigg) {\bf 1}_{\{X_i(r) \leq L(s)\}}.
\end{equation}
Note that ${\hat Z}$ is defined in the same way as $Z(r)$, except that $L(s)$ is used instead of $L(r)$ in the denominator of the sine function and in the indicator.  Lemma \ref{withina} implies that with probability tending to one as $x \rightarrow \infty$, we have $X_1(r) \leq L(r) - 2(L(r) - L(s))$.  Therefore, there are positive constants $C'$ and $C''$ such that for sufficiently large $x$,
\begin{equation}\label{ZZhat}
\P\big(C' Z(r) \leq {\hat Z} \leq C'' Z(r) \big) > 1 - \varepsilon.
\end{equation}

\begin{Lemma}\label{meanXf}
For sufficiently large $x$, we have $$\bigg| \E[X(f)|{\cal F}_r] - {\hat Z} \frac{\pi}{L(s)^2} e^{-\pi^2 B x^2/2L(s)^2} \int_0^{\infty} f(y) g(y) \: dy \bigg| < \frac{2 \pi \|f\| \varepsilon}{L(s)^2} e^{-\pi^2 B x^2/2L(s)^2} {\hat Z},$$ where $g(y) = 2y e^{-\sqrt{2} y}$ as in Theorem \ref{config1}.
\end{Lemma}

\begin{proof}
Because the right-hand side of (\ref{Dineq}) is at most $\varepsilon$ when $s = Bx^2$, it follows from Lemma \ref{fXvar} and Lemma \ref{stripdensity} that
\begin{align}
\E[X(f)|{\cal F}_r] &= \sum_{i=1}^{N(r)} \int_0^{L(s)} f(y) q_{Bx^2}(X_i(r), y) \: dy \nonumber \\
&= \frac{2 (1 + D)}{L(s)} e^{-\pi^2 B x^2/2L(s)^2} {\hat Z} \int_0^{L(s)} f(y) e^{-\sqrt{2} y} \sin \bigg( \frac{\pi y}{L(s)} \bigg) \: dy, \nonumber
\end{align}
where $|D| < \varepsilon$.  Note that
$$\lim_{x \rightarrow \infty} L(s) \int_0^{L(s)} f(y) e^{-\sqrt{2} y} \bigg| \frac{\pi y}{L(s)} - \sin \bigg( \frac{\pi y}{L(s)} \bigg) \bigg| \: dy = 0$$ and
$$\lim_{x \rightarrow \infty} \int_{L(s)}^{\infty} f(y) e^{-\sqrt{2} y} \cdot \pi y \: dy = 0.$$ It follows that
$$L(s) \int_0^{L(s)} f(y) e^{-\sqrt{2} y} \sin \bigg( \frac{\pi y}{L(s)} \bigg) \: dy = \int_0^{\infty} f(y) e^{-\sqrt{2} y} \cdot \pi y \: dy + \gamma(x),$$ where $\gamma(x) \rightarrow 0$ as $x \rightarrow \infty$.  Therefore,
\begin{equation}\label{Xfapprox}
\E[X(f)|{\cal F}_r] = {\hat Z} \frac{\pi(1 + D)}{L(s)^2} e^{-\pi^2 B x^2/2L(s)^2} \bigg( \int_0^{\infty} f(y) g(y) \: dy + \frac{2 \gamma(x)}{\pi} \bigg).
\end{equation}
To obtain the result from (\ref{Xfapprox}), first note that the error term involving $\gamma(x)$ is bounded by $2 (1 + \varepsilon) L(s)^{-2} e^{-\pi^2 B x^2/2L(s)^2}{\hat Z} \gamma(x)$, and then bound the remaining error term involving $D$ by $\pi \varepsilon L(s)^{-2} e^{-\pi^2 B x^2/2L(s)^2} \|f\| {\hat Z}$.
\end{proof}

\begin{Lemma}\label{varXf}
There is a constant $C$ such that for sufficiently large $x$,
$$\textup{Var}(X(f)|{\cal F}_r) \leq \frac{C Y(r) e^{\sqrt{2} L(s)}}{x^{11/2}}.$$
\end{Lemma}

\begin{proof}
By summing over the contributions of the particles at time $r$ and applying Lemma \ref{fXvar}, we get
\begin{align}\label{condvarX}
\mbox{Var}(X(f)|{\cal F}_r) &\leq \sum_{i=1}^{N(r)} \int_0^{L(s)} f(y)^2 \: q_{Bx^2}(X_i(r),y) \: dy \nonumber \\
&\hspace{.3in}+ 2 \sum_{i=1}^{N(r)} \int_0^{Bx^2} \int_0^{L(s)} q_u(X_i(r), z) \bigg( \int_0^{L(s)} f(y) q_{Bx^2-u}(z,y) \: dy \bigg)^2 \: dz \: du.
\end{align}
The first term on the right-hand side of (\ref{condvarX}) is bounded by $\|f\|^2 \E[X(1)|{\cal F}_r]$, where $X(1)$ denotes the value of $X(f)$ when $f(x) = 1$ for all $x$.  Consequently, by Lemma \ref{meanXf}, this term is bounded above by $C {\hat Z} x^{-2} \leq C Y(r) x^{-2} \leq C Y(r) e^{\sqrt{2} L(s)}/x^{11/2}$.

It remains to bound the second term.  The strategy is very similar to that used in the proof of Proposition 14 in \cite{bbs} and involves splitting the outer integral into four pieces.  Suppose $0 < w < L(s)$.  Using Lemma \ref{stripdensity} and equations (\ref{q1}) and (\ref{q5}),
\begin{align}\label{vt1}
&\int_0^{Bx^2/2} \int_0^{L(s)} q_u(w,z) \bigg( \int_0^{L(s)} f(y) q_{Bx^2-u}(z,y) \: dy \bigg)^2 \: dz \: du \nonumber \\
&\qquad \leq \int_0^{Bx^2/2} \int_0^{L(s)} q_u(w,z) \bigg( \int_0^{L(s)} \frac{C}{L(s)} e^{\sqrt{2} z} \sin \bigg( \frac{\pi z}{L(s)} \bigg) e^{-\sqrt{2} y} \sin \bigg( \frac{\pi y}{L(s)} \bigg) \: dy \bigg)^2 \: dz \: du \nonumber \\
&\qquad \leq \frac{C}{L(s)^2} \int_0^{Bx^2/2} \int_0^{L(s)} q_u(w,z) e^{2 \sqrt{2} z} \sin \bigg( \frac{\pi z}{L(s)} \bigg)^2 \bigg( \int_0^{L(s)} e^{-\sqrt{2} y} \sin \bigg( \frac{\pi y}{L(s)} \bigg) \: dy \bigg)^2 \: dz \: du \nonumber \\
&\qquad \leq \frac{C}{L(s)^4} \int_0^{L(s)} e^{2 \sqrt{2} z} \sin \bigg( \frac{\pi z}{L(s)} \bigg)^2 \bigg( \int_0^{Bx^2/2} q_u(w,z) \: du \bigg) \: dz. \nonumber \\
&\qquad \leq \frac{C e^{\sqrt{2} w}}{L(s)^4} \int_0^{L(s)} e^{\sqrt{2} z} \sin \bigg( \frac{\pi z}{L(s)} \bigg)^2 \frac{w(L(s) - z)}{L(s)} \: dz \nonumber \\
&\qquad \leq \frac{C e^{\sqrt{2} w} e^{\sqrt{2} L(s)}}{L(s)^6}.
\end{align}
Using Lemma \ref{stripdensity} and (\ref{q2}),
\begin{align}\label{vt2}
&\int_{Bx^2/2}^{Bx^2 - L(s)^{7/4}} \int_0^{L(s)} q_u(w,z) \bigg( \int_0^{L(s)} f(y) q_{Bx^2-u}(z,y) \: dy \bigg)^2 \: dz \: du \nonumber \\
&\qquad \leq \int_{Bx^2/2}^{Bx^2 - L(s)^{7/4}} \int_0^{L(s)} \frac{C}{L(s)} e^{\sqrt{2} w} \sin \bigg( \frac{\pi w}{L(s)} \bigg) e^{-\sqrt{2} z} \sin \bigg(\frac{\pi z}{L(s)} \bigg) \nonumber \\
&\qquad \qquad \times \bigg( \int_0^{L(s)} \frac{C}{L(s)} e^{\sqrt{2} z} \sin \bigg( \frac{\pi z}{L(s)} \bigg) e^{-\sqrt{2} y} \sin \bigg( \frac{\pi y}{L(s)} \bigg) \cdot \frac{CL(s)^3}{(Bx^2 - u)^{3/2}} \: dy \bigg)^2  \: dz \: du \nonumber \\
&\qquad \leq CL(s)^3 e^{\sqrt{2} w} \sin \bigg( \frac{\pi w}{L(s)} \bigg) \bigg( \int_{Bx^2/2}^{Bx^2 - L(s)^{7/4}} \frac{1}{(Bx^2 - u)^3} \: du \bigg) \nonumber \\
&\qquad \qquad \times \bigg( \int_0^{L(s)} e^{\sqrt{2} z} \sin \bigg( \frac{\pi z}{L(s)} \bigg)^3 \: dz \bigg) \bigg( \int_0^{L(s)} e^{-\sqrt{2} y} \sin \bigg( \frac{\pi y}{L(s)} \bigg) \: dy \bigg)^2 \nonumber \\
&\qquad \leq C L(s)^3 e^{\sqrt{2} w} \sin \bigg( \frac{\pi w}{L(s)} \bigg) \cdot \frac{1}{L(s)^{7/2}} \cdot \frac{e^{\sqrt{2} L(s)}}{L(s)^3} \cdot \frac{1}{L(s)^2} \nonumber \\
&\qquad = \frac{C e^{\sqrt{2}w} e^{\sqrt{2} L(s)}}{L(s)^{11/2}} \sin \bigg( \frac{\pi w}{L(s)} \bigg).
\end{align}
Using (\ref{q3}), we get
\begin{align}\label{vt3}
&\int_{Bx^2 - L(s)^{7/4}}^{Bx^2 - 1} \int_0^{2L(s)/3} q_u(w,z) \bigg( \int_0^{L(s)} f(y) q_{Bx^2-u}(z,y) \: dy \bigg)^2 \: dz \: du \nonumber \\
&\qquad \leq \int_{Bx^2 - L(s)^{7/4}}^{Bx^2 - 1} \int_0^{2L(s)/3} \frac{C}{L(s)} e^{\sqrt{2} w} \sin \bigg( \frac{\pi w}{L(s)} \bigg) \nonumber \\
&\qquad \qquad \times e^{-\sqrt{2} z} \sin \bigg( \frac{\pi z}{L(s)} \bigg) \bigg( \int_0^{L(s)} \frac{C e^{\sqrt{2}(z-y)}}{(Bx^2 - u)^{1/2}} \: dy \bigg)^2 \: dz \: du \nonumber \\
&\qquad \leq \frac{C}{L(s)} e^{\sqrt{2} w} \sin \bigg( \frac{\pi w}{L(s)} \bigg) \bigg( \int_{Bx^2 - L(s)^{7/4}}^{Bx^2 - 1} \frac{1}{Bx^2 - u} \: du \bigg) \bigg( \int_0^{2L(s)/3} e^{\sqrt{2} z} \sin \bigg( \frac{\pi z}{L(s)} \bigg) \: dz \bigg) \nonumber \\
&\qquad \leq \frac{C e^{\sqrt{2} w} e^{2 \sqrt{2} L(s) / 3} \log L(s)}{L(s)} \sin \bigg( \frac{\pi w}{L(s)} \bigg)
\end{align}
and
\begin{align}\label{vt4}
&\int_{Bx^2 - L(s)^{7/4}}^{Bx^2 - 1} \int_{2L(s)/3}^{L(s)} q_u(w,z) \bigg( \int_0^{L(s)} f(y) q_{Bx^2-u}(z,y) \: dy \bigg)^2 \: dz \: du \nonumber \\
&\qquad \leq \int_{Bx^2 - L(s)^{7/4}}^{Bx^2 - 1} \int_{2L(s)/3}^{L(s)} \frac{C}{L(s)} e^{\sqrt{2} w} \sin \bigg( \frac{\pi w}{L(s)} \bigg) \nonumber \\
&\qquad \qquad \times e^{-\sqrt{2} z} \sin \bigg( \frac{\pi z}{L(s)} \bigg) \bigg( \int_0^{L(s)} \frac{C e^{\sqrt{2}(z-y)} e^{-(z-y)^2/2(Bx^2 - u)}}{(Bx^2 - u)^{1/2}} \: dy \bigg)^2 \: dz \: du \nonumber \\
&\qquad \leq \frac{C}{L(s)} e^{\sqrt{2} w} \sin \bigg( \frac{\pi w}{L(s)} \bigg) \bigg( \int_{Bx^2 - L(s)^{7/4}}^{Bx^2 - 1} \frac{1}{Bx^2 - u} \: du \bigg) \nonumber \\
&\qquad \qquad \times \int_{2L(s)/3}^{L(s)} e^{\sqrt{2} z} \sin \bigg( \frac{\pi z}{L(s)} \bigg) \bigg( \int_0^{L(s)} e^{-\sqrt{2} y} e^{-(z-y)^2/2L(s)^{7/4}} \: dy \bigg)^2 \: dz \nonumber \\
&\qquad \leq \frac{C \log L(s)}{L(s)} e^{\sqrt{2} w} \sin \bigg( \frac{\pi w}{L(s)} \bigg) \nonumber \\
&\qquad \qquad \times \int_{2L(s)/3}^{L(s)} e^{\sqrt{2} z} \bigg( \int_0^{L(s)/3} e^{-\sqrt{2} y} e^{-(L(s)/3)^2/2L(s)^{7/4}} \: dy + \int_{L(s)/3}^{L(s)} e^{-\sqrt{2} y} \: dy \bigg)^2 \: dz \nonumber \\
&\qquad \leq \frac{C \log L(s)}{L(s)} e^{\sqrt{2} w} \sin \bigg( \frac{\pi w}{L(s)} \bigg) \bigg( \int_{2L(s)/3}^{L(s)} e^{\sqrt{2} z} \: dz \bigg) \bigg( e^{-L(s)^{1/4}/18} + e^{-\sqrt{2}L(s)/3} \bigg)^2 \nonumber \\
&\qquad \leq \frac{C \log L(s)}{L(s)} e^{\sqrt{2} w} e^{\sqrt{2} L(s)} e^{-L(s)^{1/4}/9} \sin \bigg( \frac{\pi w}{L(s)} \bigg).
\end{align}
Finally, using (\ref{q4}),
\begin{align}\label{vt5}
&\int_{Bx^2 - 1}^{Bx^2} \int_0^{L(s)} q_u(w,z) \bigg( \int_0^{L(s)} f(y) q_{Bx^2-u}(z,y) \: dy \bigg)^2 \: dz \: du \nonumber \\
&\qquad \leq \int_{Bx^2 - 1}^{Bx^2} \int_0^{L(s)} \frac{C}{L(s)} e^{\sqrt{2} w} \sin \bigg( \frac{\pi w}{L(s)} \bigg) e^{-\sqrt{2} z} \sin \bigg( \frac{\pi z}{L(s)} \bigg) \big( \|f\| e \big)^2 \: dz \: du \nonumber \\
&\qquad \leq \frac{C e^{\sqrt{2} w}}{L(s)^2} \sin \bigg( \frac{\pi w}{L(s)} \bigg).
\end{align}
The expressions in (\ref{vt1}), (\ref{vt2}), (\ref{vt3}), (\ref{vt4}), and (\ref{vt5}) are all bounded by $C e^{\sqrt{2} w} e^{\sqrt{2} L(s)}/L(s)^{11/2}$.  Because $L(s)$ and $x$ are the same to within a constant factor, we get after summing over the positions of the particles at time $r$ that the second term on the right-hand side of (\ref{condvarX}) is bounded by $C Y(r) e^{\sqrt{2} L(s)}/x^{11/2}$.  The result follows.
\end{proof}

\begin{Lemma}\label{meanXphi}
For sufficiently large $x$, we have
$$\bigg| \E[X'(\phi)|{\cal F}_r] - \frac{4{\hat Z}}{\pi} e^{-\pi^2 B x^2/2L(s)^2} \int_0^1 \phi(y) h(y) \: dy \bigg| < \frac{4 \|\phi\| \varepsilon}{\pi} e^{-\pi^2 B x^2/2L(s)^2} {\hat Z},$$
where $h(y) = \frac{\pi}{2} \sin(\pi y)$ as in Theorem \ref{config2}.
\end{Lemma}

\begin{proof}
Because the right-hand side of (\ref{Dineq}) is at most $\varepsilon$ when $s = Bx^2$, it follows from Lemma \ref{fXvar} and Lemma \ref{stripdensity} that
\begin{align}
\E[X'(\phi)|{\cal F}_r] &= \sum_{i=1}^{N(r)} \int_0^{L(s)} e^{\sqrt{2} y} \phi \bigg( \frac{y}{L(s)} \bigg) q_{Bx^2}(X_i(r), y) \: dy \nonumber \\
&= \frac{2(1 + D)}{L(s)} e^{-\pi^2 Bx^2/2L(s)^2} {\hat Z} \int_0^{L(s)} \phi \bigg( \frac{y}{L(s)} \bigg) \sin \bigg( \frac{\pi y}{L(s)} \bigg) \: dy \nonumber \\
&= \frac{4(1 + D)}{\pi} e^{-\pi^2 Bx^2/2L(s)^2} {\hat Z} \int_0^1 \phi(y) h(y) \: dy \nonumber \
\end{align}
where $|D| < \varepsilon$.  Because $h$ is a probability density, the error term involving $D$ is bounded by $(4\|\phi\| \varepsilon/\pi) e^{-\pi^2 B x^2/2L(s)^2} {\hat Z}$, as claimed.
\end{proof}

\begin{Lemma}\label{varXphi}
There is a constant $C$ such that for sufficiently large $x$,
$$\textup{Var}(X'(\phi)|{\cal F}_r) \leq \frac{C Y(r) e^{\sqrt{2} L(s)} \log x}{x^2}.$$
\end{Lemma}

\begin{proof}
By summing over the contributions of the particles at time $r$ and applying Lemma \ref{fXvar}, we get
\begin{align}\label{condvarphi}
\mbox{Var}(X'(\phi)|{\cal F}_r) &\leq \sum_{i=1}^{N(r)} \int_0^{L(s)} e^{2 \sqrt{2} y} \phi\bigg( \frac{y}{L(s)} \bigg)^2 q_{Bx^2}(X_i(r), y) \: dy \nonumber \\
&\qquad + 2 \sum_{i=1}^{N(r)} \int_0^{Bx^2} q_u(X_i(r), z) \bigg( \int_0^{L(s)} e^{\sqrt{2} y} \phi\bigg( \frac{y}{L(s)} \bigg) q_{Bx^2 - u}(z,y) \: dy \bigg)^2 \: dz \: du
\end{align}

To bound the first term on the right-hand side of (\ref{condvarphi}), note that if $0 < w < L(s)$, then, by Lemma \ref{stripdensity} and (\ref{q1}),
\begin{align}\label{pt1}
\int_0^{L(s)} e^{2 \sqrt{2} y} \phi\bigg( \frac{y}{L(s)} \bigg)^2 q_{Bx^2}(w, y) \: dy
&\leq \frac{C e^{\sqrt{2} w}}{L(s)} \sin \bigg( \frac{\pi w}{L(s)} \bigg) \int_0^{L(s)} e^{\sqrt{2} y} \sin \bigg( \frac{\pi y}{L(s)} \bigg) \: dw \nonumber \\
&\leq \frac{C e^{\sqrt{2} w} e^{\sqrt{2} L(s)}}{L(s)^2} \sin \bigg( \frac{\pi w}{L(s)} \bigg).
\end{align}
We bound the second term on the right-hand side of (\ref{condvarphi}) by breaking the outer integral into two pieces.  Using (\ref{q5}), if $0 < w < L(s)$, then
\begin{align}\label{pt2}
&\int_0^{Bx^2/2} \int_0^{L(s)} q_u(w,z) \bigg( \int_0^{L(s)} e^{\sqrt{2} y} \phi\bigg( \frac{y}{L(s)} \bigg) q_{Bx^2-u}(z,y) \: dy \bigg)^2 \: dz \: du \nonumber \\
&\qquad \leq \int_0^{Bx^2/2} \int_0^{L(s)} q_u(w,z) \bigg( \int_0^{L(s)} \frac{C}{L(s)} e^{\sqrt{2} z} \sin \bigg( \frac{\pi z}{L(s)} \bigg) \sin \bigg( \frac{\pi y}{L(s)} \bigg) \: dy \bigg)^2 \: dz \: du \nonumber \\
&\qquad \leq C \int_0^{Bx^2/2} \int_0^{L(s)} q_u(w,z) e^{2 \sqrt{2} z} \sin \bigg( \frac{\pi z}{L(s)} \bigg)^2 \: dz \: du \nonumber \\
&\qquad \leq C \int_0^{L(s)} e^{\sqrt{2} w} e^{\sqrt{2} z} \sin \bigg( \frac{\pi z}{L(s)} \bigg)^2 \frac{w(L(s) - z)}{L(s)} \: dz \nonumber \\
&\qquad \leq \frac{C e^{\sqrt{2} w} e^{\sqrt{2} L(s)}}{L(s)^2}.
\end{align}

Furthermore, by using (\ref{q6}) in the third line, making the substitution $v = Bx^2 - u$ in the fourth line, and breaking the inner integral into the piece from $0$ to $1$ and the piece from $1$ to $Bx^2/2$ in the fifth line, we get
\begin{align}\label{pt3}
&\int_{Bx^2/2}^{Bx^2} \int_0^{L(s)} q_u(w,z) \bigg( \int_0^{L(s)} e^{\sqrt{2} y} \phi\bigg( \frac{y}{L(s)} \bigg) q_{Bx^2-u}(z,y) \: dy \bigg)^2 \: dz \: du \nonumber \\
&\qquad \leq \int_{Bx^2/2}^{Bx^2} \int_0^{L(s)} \frac{C}{L(s)} e^{\sqrt{2} w} \sin \bigg( \frac{\pi w}{L(s)} \bigg) e^{-\sqrt{2} z} \sin \bigg( \frac{\pi z}{L(s)} \bigg) \bigg( \int_0^{L(s)} e^{\sqrt{2} y} q_{Bx^2 - u}(z,y) \: dy \bigg)^2 \: dz \: du \nonumber \\
&\qquad \leq \frac{C e^{\sqrt{2} w}}{L(s)} \sin \bigg( \frac{\pi w}{L(s)} \bigg) \int_{Bx^2/2}^{Bx^2} \int_0^{L(s)} e^{\sqrt{2} z} \sin \bigg( \frac{\pi z}{L(s)} \bigg) \min \bigg\{1, \frac{(L(s) - z)^2}{Bx^2 - u} \bigg\} \: dz \: du \nonumber \\
&\qquad \leq \frac{C e^{\sqrt{2} w}}{L(s)} \sin \bigg( \frac{\pi w}{L(s)} \bigg) \int_0^{L(s)} e^{\sqrt{2} z} \sin \bigg( \frac{\pi z}{L(s)} \bigg) \bigg( \int_0^{Bx^2/2} \min \bigg\{1, \frac{(L(s) - z)^2}{v} \bigg\} \: dv \bigg) \: dz \nonumber \\
&\qquad \leq \frac{C e^{\sqrt{2} w}}{L(s)} \sin \bigg( \frac{\pi w}{L(s)} \bigg) \int_0^{L(s)} e^{\sqrt{2} z} \sin \bigg( \frac{\pi z}{L(s)} \bigg) \big(1 + (L(s) - z)^2 \log x \big) \: dz \nonumber \\
&\qquad \leq \frac{C e^{\sqrt{2} w} e^{\sqrt{2} L(s)} \log x}{L(s)^2} \sin \bigg( \frac{\pi w}{L(s)} \bigg).
\end{align}
The expressions in (\ref{pt1}), (\ref{pt2}), and (\ref{pt3}) are all bounded by $(C e^{\sqrt{2} w} e^{\sqrt{2} L(s)} \log x)/L(s)^2$.  By summing over the positions of the particles at time $r$, we get that the right-hand side of (\ref{condvarphi}) is bounded by $(CY(r) e^{\sqrt{2} L(s)} \log x)/x^2$, which implies the result.
\end{proof}

\subsection{Proofs of Theorems \ref{numthm}, \ref{config1}, and \ref{config2}}

In this subsection, we use the results of Section \ref{momfunsec} to prove Theorems \ref{numthm}, \ref{config1}, and \ref{config2}.

\begin{proof}[Proof of Theorem \ref{numthm}]
Let $\kappa = 1$.  Choose $B$ as at the beginning of Section \ref{momfunsec}.  Choose $s \in [(B + 3 \kappa)x^2, (1 - \delta)t]$, and let $r = s - Bx^2$ as in Section \ref{momfunsec}.  Throughout the proof, the constants $C$, $C'$, and $C''$ will be allowed to depend on $B$, $\delta$ and $\varepsilon$. Recall that $X(1)$ denotes the value of $X(f)$ when $f(x) = 1$ for all $x$.  By Lemma \ref{fphi},
\begin{equation}\label{XN}
\P(X(1) = N(s)) > 1 - 3 \varepsilon
\end{equation}
for sufficiently large $x$.  By Lemma \ref{meanXf},
\begin{equation}\label{955}
(1 - 2 \varepsilon) {\hat Z} \frac{\pi}{L(s)^2} e^{-\pi^2 B x^2/2L(s)^2} \leq \E[X(1)|{\cal F}_r] \leq (1 + 2 \varepsilon) {\hat Z} \frac{\pi}{L(s)^2} e^{-\pi^2 B x^2/2L(s)^2}.
\end{equation}
Using Lemma \ref{varXf} and the conditional Chebyshev's Inequality,
\begin{equation}\label{cheb}
\P \bigg(\big|X(1) - \E[X(1)|{\cal F}_r] \big| > \frac{1}{2} \E[X(1)|{\cal F}_r] \bigg| {\cal F}_r \bigg) \leq \frac{C Y(r) e^{\sqrt{2} L(s)}}{x^{11/2} \E[X(1)|{\cal F}_r]^2} \leq \frac{C Y(r) e^{\sqrt{2} L(s)}}{x^{3/2} {\hat Z}^2}.
\end{equation}
By (\ref{ZZhat}), Proposition \ref{Zlower}, Proposition \ref{Zupper}, and Proposition \ref{Yupper}, there are constants $C$, $C'$ and $C''$ such that with probability at least $1 - 4 \varepsilon$, we have
\begin{equation}\label{newZ}
C' x^{-1} \exp \big((3 \pi^2)^{1/3}(t - r)^{1/3} \big) \leq {\hat Z} \leq C'' x^{-1} \exp \big((3 \pi^2)^{1/3}(t - r)^{1/3} \big)
\end{equation}
and
\begin{equation}\label{newY}
Y(r) \leq C x^{-1} \exp \big((3 \pi^2)^{1/3}(t - r)^{1/3} \big).
\end{equation}
Thus, on an event of probability at least $1 - 4 \varepsilon$, the quantity on the right-hand side of (\ref{cheb}) is bounded above by $$C x^{-1/2} \exp \big(\sqrt{2} L(s) - (3 \pi^2)^{1/3}(t - r)^{1/3} \big) = C x^{-1/2} \exp \big( (3 \pi^2)^{1/3}(t - s)^{1/3} - (3 \pi^2)^{1/3}(t - r)^{1/3} \big),$$ which tends to zero as $x \rightarrow \infty$ because the exponential term is bounded by a constant.  By (\ref{955}), on this same event of probability $1 - 4 \varepsilon$, there are constants $C'$ and $C''$ such that $$C' x^{-3} \exp \big((3 \pi^2)^{1/3}(t - s)^{1/3} \big) \leq \frac{1}{2} \E[X(1)|{\cal F}_r] \leq \frac{3}{2} \E[X(1)|{\cal F}_r] \leq C'' x^{-3} \exp \big((3 \pi^2)^{1/3}(t - s)^{1/3} \big).$$  Combining these results with (\ref{XN}), we get $$\P \bigg( C' x^{-3} \exp \big((3 \pi^2)^{1/3}(t - s)^{1/3} \big) \leq N(s) \leq C'' x^{-3} \exp \big((3 \pi^2)^{1/3}(t - s)^{1/3} \big) \bigg) > 1 - 7 \varepsilon$$ for sufficiently large $x$.  Because the constants $C'$ and $C''$ do not depend on $s$ and
$$(3 \pi^2)^{1/3} (t - s)^{1/3} = \sqrt{2} \bigg(1 - \frac{s}{\tau x^3} \bigg)^{1/3} x,$$ the result follows.
\end{proof}

We will now prove slightly more general versions of Theorems \ref{config1} and \ref{config2}.  The generalizations will be useful later for the proof of Theorem \ref{rtthm}.  Instead of setting $s = u x^3$ and letting $x \rightarrow \infty$, we consider a sequence $(x_n)_{n=1}^{\infty}$ tending to infinity, and then a sequence of times $(s_n)_{n=1}^{\infty}$ such that $s_n \sim u x_n^3$.  We define $X_i(s)$, $N(s)$, $X(f)$, and $X'(\phi)$ as before, but with the initial particle located at $x_n$ rather than at $x$.  Note that $X_i(s)$, $N(s)$, $X(f)$, and $X'(\phi)$ depend on $n$, even though this dependence is not recorded in the notation.  The following proposition implies Theorem \ref{config1}.  Here $\sim$ means that the ratio of the two sides tends to one as $n \rightarrow \infty$.

\begin{Prop}
Suppose $0 < u < \tau$.  Consider a sequence of times $(s_n)_{n=1}^{\infty}$ such that $s_n \sim u x_n^3$.  Let $$\chi_n(u) = \frac{1}{N(s_n)} \sum_{i=1}^{N(s_n)} \delta_{X_i(s_n)}.$$  Define $\mu$ as in Theorem \ref{config1}.  Then $\chi_n(u) \Rightarrow \mu$ as $n \rightarrow \infty$.
\end{Prop}

\begin{proof}
To show that $\chi_n(u) \Rightarrow \mu$ as $n \rightarrow \infty$, it suffices to show (see, for example, Theorem 16.16 of \cite{kall}) that for all bounded continuous functions $f: [0, \infty) \rightarrow \R$, we have
\begin{equation}\label{fconvprob}
\frac{1}{N(s_n)} \sum_{i=1}^{N(s_n)} f(X_i(s_n)) \rightarrow_p \int_0^{\infty} g(y) f(y) \: dy,
\end{equation}
where $g(y) = 2y e^{-\sqrt{2} y}$ for $y \geq 0$ and $\rightarrow_p$ denotes convergence in probability as $n \rightarrow \infty$.

Fix a bounded continuous function $f: [0, \infty) \rightarrow \R$.  Let $\varepsilon > 0$, and choose $B$ as at the beginning of 
Section \ref{momfunsec}.  Let $r_n = s_n - Bx^2$.  By Lemma \ref{fphi}, for sufficiently large $n$,
\begin{equation}\label{XfX1}
\P \bigg( \frac{1}{N(s_n)} \sum_{i=1}^{N(s_n)} f(X_i(s_n)) = \frac{X(f)}{X(1)} \bigg) > 1 - 3 \varepsilon.
\end{equation}

By Lemma \ref{varXf} and the conditional Chebyshev's Inequality,
\begin{align}\label{Xfcheb}
\P \bigg( \big|X(f) - \E[X(f)|{\cal F}_{r_n}] \big| > x_n^{-19/6} e^{\sqrt{2} L(s_n)} \bigg| {\cal F}_{r_n} \bigg) &\leq \frac{C Y(r_n) e^{\sqrt{2} L(s_n)}}{x_n^{11/2}} \cdot \frac{x_n^{19/3}}{e^{2 \sqrt{2} L(s_n)}} \nonumber \\
&\leq \frac{C Y(r_n) x_n^{5/6}}{e^{\sqrt{2} L(s_n)}}.
\end{align}
Both (\ref{newZ}) and (\ref{newY}) hold, with $r_n$ in place of $r$, with probability at least $1 - 4 \varepsilon$ for sufficiently large $n$.  Because $(t - r_n)^{1/3} - (t - s_n)^{1/3}$ is bounded by a constant, the expression obtained by replacing $Y(r_n)$ on the right-hand side of (\ref{Xfcheb}) by the upper bound from (\ref{newY}) tends to zero as $n \rightarrow \infty$, and thus is less than $\varepsilon$ for sufficiently large $n$.  The same convergence holds with $X(1)$ in place of $X(f)$ on the left-hand side of (\ref{Xfcheb}).  Thus, for sufficiently large $n$, on an event of probability at least $1 - 5 \varepsilon$, we have
$$\frac{\E[X(f)|{\cal F}_{r_n}] - x_n^{-19/6} e^{\sqrt{2} L(s_n)}}{\E[X(1)|{\cal F}_{r_n}] + x_n^{-19/6} e^{\sqrt{2} L(s_n)}} \leq \frac{X(f)}{X(1)} \leq \frac{\E[X(f)|{\cal F}_{r_n}] + x_n^{-19/6} e^{\sqrt{2} L(s_n)}}{\E[X(1)|{\cal F}_{r_n}] - x_n^{-19/6} e^{\sqrt{2} L(s_n)}}.$$
This inequality, when combined with Lemma \ref{meanXf}, becomes
\begin{align}
&\frac{{\hat Z} \pi L(s_n)^{-2} e^{-\pi^2 B x_n^2/2L(s_n)^2} (\int_0^{\infty} f(y) g(y) \: dy - 2 \|f\| \varepsilon) - x_n^{-19/6} e^{\sqrt{2} L(s_n)}}{{\hat Z} \pi L(s_n)^{-2} e^{-\pi^2 B x_n^2/2 L(s_n)^2} (1 + 2 \varepsilon) + x_n^{-19/6} e^{\sqrt{2}L(s_n)}} \nonumber \\
&\qquad \leq
\frac{X(f)}{X(1)} \leq \frac{{\hat Z} \pi L(s_n)^{-2} e^{-\pi^2 B x_n^2/2L(s_n)^2} (\int_0^{\infty} f(y) g(y) \: dy + 2 \|f\| \varepsilon) + x_n^{-19/6} e^{\sqrt{2} L(s_n)}}{{\hat Z} \pi L(s_n)^{-2} e^{-\pi^2 B x_n^2/2 L(s_n)^2} (1 - 2 \varepsilon) - x_n^{-19/6} e^{\sqrt{2}L(s_n)}}. \nonumber
\end{align}
When (\ref{newZ}) holds, we have $x_n^{-3} e^{\sqrt{2} L(s_n)} \leq C {\hat Z} L(s_n)^{-2}$, and thus for sufficiently large $n$, $$x_n^{-19/6} e^{\sqrt{2} L(s_n)} \leq {\hat Z} \pi L(s_n)^{-2} e^{-\pi^2 B x_n^2/2L(s_n)^2} \varepsilon.$$  Therefore, for sufficiently large $n$,
$$\frac{1}{1 + 3 \varepsilon} \bigg( \int_0^{\infty} f(y) g(y) \: dy - 2 \|f\| \varepsilon - \varepsilon \bigg) \leq \frac{X(f)}{X(1)} \leq \frac{1}{1 - 3 \varepsilon} \bigg( \int_0^{\infty} f(y) g(y) \: dy + 2 \|f\| \varepsilon + \varepsilon \bigg)$$
with probability at least $1 - 5 \varepsilon$.  In view of (\ref{XfX1}), we can let $\varepsilon \rightarrow 0$ to obtain (\ref{fconvprob}).
\end{proof}

The following proposition implies Theorem \ref{config2}.

\begin{Prop}
Suppose $0 < u < \tau$.  Consider a sequence of times $(s_n)_{n=1}^{\infty}$ such that $s_n \sim ux_n^3$ as $n \rightarrow \infty$.  Let $$\eta_n(u) = \frac{1}{Y(s_n)} \sum_{i=1}^{N(s_n)} e^{\sqrt{2} X_i(s_n)} \delta_{X_i(s_n)/L(s_n)}.$$  Let $\nu$ be defined as in Theorem \ref{config2}.  Then $\eta_n(u) \Rightarrow \nu$ as $n \rightarrow \infty$.
\end{Prop}

\begin{proof}
The proof is very similar to the proof of Theorem \ref{config1}.
It suffices to show that we have $\P(X_1(s_n) < L(s_n)) \rightarrow 1$ as $n \rightarrow \infty$, and that for all bounded continuous functions $\phi: [0, 1] \rightarrow \R$,
\begin{equation}\label{phiprob}
\frac{1}{Y(s_n)} \sum_{i=1}^{N(s_n)} e^{\sqrt{2} X_i(s_n)} \phi \bigg( \frac{X_i(s_n)}{L(s_n)} \bigg) \rightarrow_p \int_0^1 \phi(y) h(y) \: dy.
\end{equation}
That $\P(X_1(s_n) < L(s_n)) \rightarrow 1$ as $n \rightarrow \infty$ follows immediately from Lemma \ref{withina} with $a = 0$.

Fix a bounded continuous function $\phi: [0,1] \rightarrow \R$.  Let $\varepsilon > 0$, and choose $B$ as at the beginning of Section \ref{momfunsec}.  Let $r_n = s_n - Bx_n^2$.  Let $X'(1)$ denote the value of $X'(\phi)$ when $\phi(x) = 1$ for all $x \in [0,1]$.  By Lemma \ref{fphi}, for sufficiently large $x$,
\begin{equation}\label{YXprime}
\P \bigg( \frac{1}{Y(s_n)} \sum_{i=1}^{N(s_n)} e^{\sqrt{2} X_i(s_n)} \phi \bigg( \frac{X_i(s_n)}{L(s_n)} \bigg) = \frac{X'(\phi)}{X'(1)} \bigg) > 1 - 3 \varepsilon.
\end{equation}

By Lemma \ref{varXphi} and the conditional Chebyshev's Inequality,
\begin{align}\label{Xphicheb}
\P \bigg( \big|X'(\phi) - \E[X'(\phi)|{\cal F}_{r_n}] \big| > x_n^{-4/3} e^{\sqrt{2} L(s_n)} \bigg| {\cal F}_{r_n} \bigg) &\leq \frac{C Y(r_n) e^{\sqrt{2} L(s_n)} \log x_n}{x_n^2} \cdot \frac{x_n^{8/3}}{e^{2 \sqrt{2} L(s_n)}} \nonumber \\
&\leq \frac{C Y(r_n) x_n^{2/3} \log x_n}{e^{\sqrt{2} L(s_n)}}.
\end{align}
Recall that (\ref{newZ}) and (\ref{newY}) both hold with probability at least $1 - 4 \varepsilon$ for sufficiently large $n$.  The expression obtained by replacing $Y(r_n)$ with the right-hand side of (\ref{newY}) on the right-hand side of (\ref{Xphicheb}) tends to zero as $x_n \rightarrow \infty$, and the same result holds when $X'(\phi)$ is replaced by $X'(1)$ on the left-hand side.  Thus, for sufficiently large $n$, on an event of probability at least $1 - 5 \varepsilon$, we have
$$\frac{\E[X'(\phi)|{\cal F}_{r_n}] - x_n^{-4/3} e^{\sqrt{2} L(s_n)}}{\E[X'(1)|{\cal F}_{r_n}] + x_n^{-4/3} e^{\sqrt{2} L(s_n)}} \leq \frac{X'(\phi)}{X'(1)} \leq \frac{\E[X'(\phi)|{\cal F}_{r_n}] + x_n^{-4/3} e^{\sqrt{2} L(s_n)}}{\E[X'(1)|{\cal F}_{r_n}] - x_n^{-4/3} e^{\sqrt{2} L(s_n)}}.$$
Combining this inequality with Lemma \ref{meanXphi} gives
\begin{align}
&\frac{4 \pi^{-1} {\hat Z} e^{-\pi^2Bx_n^2/2L(s_n)^2}(\int_0^1 \phi(y) h(y) \: dy - \|\phi\| \varepsilon) - x_n^{-4/3} e^{\sqrt{2}L(s_n)}}{4 \pi^{-1} {\hat Z} e^{-\pi^2Bx_n^2/2L(s_n)^2}(1 + \varepsilon) + x_n^{-4/3} e^{\sqrt{2} L(s_n)}} \nonumber \\
&\qquad \leq \frac{X'(\phi)}{X'(1)} \leq \frac{4 \pi^{-1} {\hat Z} e^{-\pi^2Bx_n^2/2L(s_n)^2}(\int_0^1 \phi(y) h(y) \: dy + \|\phi\| \varepsilon) + x_n^{-4/3} e^{\sqrt{2}L(s_n)}}{4 \pi^{-1} {\hat Z} e^{-\pi^2Bx_n^2/2L(s_n)^2}(1 - \varepsilon) - x_n^{-4/3} e^{\sqrt{2} L(s_n)}}. \nonumber
\end{align}
Because $x_n^{-1} e^{\sqrt{2} L(s_n)} \leq C {\hat Z}$ when (\ref{newZ}) holds, we have $x_n^{-4/3} e^{\sqrt{2} L(s_n)} \leq 4 \pi^{-1} {\hat Z} e^{-\pi^2Bx_n^2/2L(s_n)^2} \varepsilon$ for sufficiently large $n$ when (\ref{newZ}) holds.  Therefore, for sufficiently large $n$,
$$\frac{1}{1 + 2 \varepsilon} \bigg( \int_0^1 \phi(y) h(y) \: dy - \|\phi\| \varepsilon - \varepsilon \bigg) \leq \frac{X'(\phi)}{X'(1)} \leq \frac{1}{1 - 2 \varepsilon} \bigg( \int_0^1 \phi(y) h(y) \: dy + \|\phi\| \varepsilon + \varepsilon \bigg)$$
with probability at least $1 - 5 \varepsilon$.  In view of (\ref{YXprime}), we can let $\varepsilon \rightarrow 0$ to obtain (\ref{phiprob}).
\end{proof}

\section{Position of the right-most particle}\label{rtsec}

In this section, we prove Theorem \ref{rtthm}.  Consider branching Brownian motion without killing and with a drift of $-\sqrt{2}$.  Let $u(t,w)$ be the probability that if at time zero there is a single particle at the origin, then the position of the right-most particle at time $t$ will be greater than or equal to $w$. Define $m(t) = \inf\{w : u(t,w) \ge 1/2\}$.
By Proposition 8.2 on page 127 of \cite{bram83}, applied with $y_0 = -1$, and by Corollary 1 on page 130 of \cite{bram83}, applied with $\alpha(r,t) = -1$, there exist positive constants $T$, $C'$, $C''$, and $C_{12}$ such that if $t \geq T$, then
\begin{equation}\label{bramupper}
u(t,w) \leq C'' e^t \int_{-1}^0 \frac{e^{-(w + \sqrt{2}t - z)^2/2t}}{\sqrt{2 \pi t}} \big(1 - e^{-2(z+1)(w - m(t))/t} \big) \: dz
\end{equation}
and
\begin{equation}\label{bramlower}
u(t,w) \geq C' e^t \int_{-1}^0 \frac{e^{-(w + \sqrt{2}t - z)^2/2t}}{\sqrt{2 \pi t}} \big(1 - e^{-2(z+1)(w - m(t))/t} \big) \: dz
\end{equation}
for all $w \geq m(t) + 1$, where
\begin{equation}\label{735}
\bigg| m(t) + \frac{3}{2 \sqrt{2}} \log t \bigg| \leq C_{12}.
\end{equation}
See (8.4) and (8.18) of \cite{bram83} for the bounds on $m(t)$, and observe that $m(t)$ here corresponds to $m_{1/2}(t) - \sqrt{2} t$ in the notation of \cite{bram83}.

\begin{Lemma}\label{u1lem}
Suppose $0 < \gamma \leq 1$.  Suppose that $t = \gamma x^2$ and that $w = -(3/2 \sqrt{2}) \log t + y$, where $1 + C_{12} \leq y \leq C_{13} x$ for some positive constant $C_{13}$.  Then there exists $x_0 > 0$, depending on $\gamma$, such that for $x \geq x_0$, $$C' y e^{-\sqrt{2} y} e^{-y^2/2t} \leq u(t,w) \leq C'' y e^{-\sqrt{2} y} e^{-y^2/2t},$$ where $C'$ and $C''$ are positive constants that do not depend on $\gamma$.
\end{Lemma}

\begin{Rmk}
{\em We note that similar bounds on $u$ may be obtained directly by PDE methods, and these have in fact been used in \cite{hnrr1} to reprove Bramson's logarithmic correction result of \cite{bram83} and to extend it to the setup of periodic branching rates (see \cite{hnrr2}).}
\end{Rmk}

\begin{proof}
We may assume that $x$ is large enough that $t \geq \max\{1, T\}$.  If $-1 \leq z \leq 0$, then using (\ref{735}),
\begin{equation}\label{zw1}
\frac{2(z+1)(w - m(t))}{t} \leq \frac{2(y - (3/2\sqrt{2}) \log t - m(t))}{t} \leq \frac{2(y + C_{12})}{t} \leq \frac{4y}{t}.
\end{equation}
It follows that
\begin{equation}\label{zwupper}
1 - e^{-2(z+1)(w - m(t))/t} \leq \frac{4y}{t}.
\end{equation}
Because $y \leq C_{13}x$ and $t = \gamma x^2$, the expression in (\ref{zw1}) tends to zero as $x \rightarrow \infty$.
Therefore, if $-1/2 \leq z \leq 0$, we have, for sufficiently large $x$,
\begin{equation}\label{zwlower}
1 - e^{-2(z+1)(w - m(t))/t} \geq \frac{1}{2} \cdot \frac{2(z+1)(w - m(t))}{t} \geq \frac{y - (3/2\sqrt{2}) \log t - m(t)}{2t} \geq \frac{y - C_{12}}{2t} \geq \frac{Cy}{t}.
\end{equation}

Next, observe that $$e^{-(w + \sqrt{2} t - z)^2/2t} = e^{-(y-z)^2/2t} e^{(3/2\sqrt{2}) (y - z) (\log t)/t} e^{-9(\log t)^2/16t} e^{-\sqrt{2}(y - z)} e^{-t} t^{3/2}.$$
If $-1 \leq z \leq 0$, then $e^{-\sqrt{2}} \leq e^{\sqrt{2} z} \leq 1$.  Also, $e^{-9(\log t)^2/16t}$ tends to one as $x \rightarrow \infty$.  Furthermore, because $t = \gamma x^2$ and $y \leq C_{13}x$, we have
$e^{(3/2\sqrt{2}) (y - z) (\log t)/t} \rightarrow 1$ and $e^{-(y - z)^2/2t}/e^{-y^2/2t} \rightarrow 1$ as $x \rightarrow \infty$.  It follows that there exists $x_0 > 0$, depending on $\gamma$, and positive constants $C'$ and $C''$ such that if $x \geq x_0$, then
\begin{equation}\label{expbd}
C' e^{-y^2/2t} e^{-\sqrt{2} y} e^{-t} t^{3/2} \leq e^{-(w + \sqrt{2} t - z)^2/2t} \leq C'' e^{-y^2/2t} e^{-\sqrt{2} y} e^{-t} t^{3/2}.
\end{equation}
Combining (\ref{bramupper}), (\ref{zwupper}), and (\ref{expbd}), we get that for sufficiently large $x$,
\begin{align}\label{ufin1}
u(t,w) &\leq C e^t \int_{-1}^0 \frac{e^{-(w + \sqrt{2}t - z)^2/2t}}{\sqrt{2 \pi t}} \big(1 - e^{-2(z+1)(w - m(t))/t} \big) \: dz \nonumber \\
&\leq C e^t \int_{-1}^0 \frac{e^{-y^2/2t} e^{-\sqrt{2} y} e^{-t} t^{3/2}}{\sqrt{2 \pi t}} \cdot \frac{y}{t} \: dz \nonumber \\
&\leq C y e^{-\sqrt{2} y} e^{-y^2/2t}.
\end{align}
By similar reasoning using (\ref{bramlower}), (\ref{zwlower}), and (\ref{expbd}), we get that for sufficiently large $x$,
\begin{align}\label{ufin2}
u(t,w) &\geq C e^t \int_{-1/2}^0 \frac{e^{-(w + \sqrt{2}t - z)^2/2t}}{\sqrt{2 \pi t}} \big(1 - e^{-2(z+1)(w - m(t))/t} \big) \: dz \nonumber \\
&\geq C e^t \int_{-1/2}^0 \frac{e^{-y^2/2t} e^{-\sqrt{2} y} e^{-t} t^{3/2}}{\sqrt{2 \pi t}} \cdot \frac{y}{t} \: dz \nonumber \\
&\geq C y e^{-\sqrt{2} y} e^{-y^2/2t}.
\end{align}
The result follows from (\ref{ufin1}) and (\ref{ufin2}).
\end{proof}

\begin{Lemma}\label{u2lem}
Suppose $0 < \gamma \leq 1$.  Suppose $t \leq \gamma x^2$ and $w \geq C_{14} x$ for some positive constant $C_{14}$.  Then there exists $x_0 > 0$, depending on $\gamma$, such that for $x \geq x_0$,
\begin{equation}\label{ut2}
u(t,w) \leq C \gamma^{-3/2} x^{-3} w e^{-\sqrt{2} w} e^{-C_{15}/\gamma}
\end{equation}
for some positive constants $C$ and $C_{15}$ that do not depend on $\gamma$.
\end{Lemma}

\begin{proof}
If $-1 \leq z \leq 0$, then
\begin{equation}\label{1281}
1 - e^{-2(z+1)(w - m(t))/t} \leq \frac{2(z+1)(w - m(t))}{t} \leq \frac{Cw}{t}.
\end{equation}
Also, for sufficiently large $x$,
\begin{equation}\label{1282}
e^{-(w + \sqrt{2}t - z)^2/2t} = e^{-t} e^{-\sqrt{2}(w-z)} e^{-(w-z)^2/2t} \leq C e^{-t} e^{-\sqrt{2} w} e^{-C_{14}^2 x^2/t}.
\end{equation}
By (\ref{bramupper}), (\ref{1281}), and (\ref{1282}), we get that when $T \leq t \leq \gamma x^2$, $$u(t,w) \leq Cw e^{-\sqrt{2} w} t^{-3/2} e^{-C_{14}^2 x^2/t}.$$ The function $t \mapsto t^{-3/2} e^{-C_{14}^2 x^2/t}$ is increasing when
$t \leq (2 C_{14}^2 x^2)/3$ which means that for $\gamma \leq 2 C_{14}^2/3$, we have $$u(t,w) \leq C \gamma^{-3/2} x^{-3} w e^{-\sqrt{2} w} e^{-C_{14}^2/2 \gamma}$$ whenever $T \leq t \leq \gamma x^2$.  This is enough to imply (\ref{ut2}) except in the case when $t < T$.  However, when $t < T$, by the Many-to-One Lemma and Markov's Inequality, $u(t,w)$ is bounded above by $e^t$ times the probability that an individual Brownian particle started at the origin is to the right of $w$ by time $t$.  For the purpose of obtaining an upper bound on $u(t,w)$, we may ignore the drift of $-\sqrt{2}$.
Therefore, using that $$\int_z^{\infty} e^{-x^2/2} \: dx \leq z^{-1} e^{-z^2/2},$$ we have
$$u(t,w) \leq e^t \int_{w/\sqrt{t}}^{\infty} \frac{1}{\sqrt{2 \pi}} e^{-x^2/2} \: dx \leq \frac{e^t \sqrt{t}}{\sqrt{2 \pi} w} e^{-w^2/2t} \leq \frac{e^T T}{\sqrt{2 \pi} w} e^{-w^2/2T}.$$  Because $w \geq C_{14} x$, this expression is bounded above by the right-hand side of (\ref{ut2}) for $x \geq x_0$, where $x_0$ depends on $\gamma$.
\end{proof}

We now return to the setting of Theorem \ref{rtthm}, in which there is initially a particle at $x$ and particles are killed when they reach the origin.

\begin{Lemma}\label{Dlem}
Let $\varepsilon > 0$.  Let $0 < u < \tau$, and let $s = ux^3$.  Let $\gamma > 0$.  Let $D$ be the number of particles that are killed at the origin between times $s - \gamma x^2$ and $s$.  Then there exists a positive constant $C$, depending on $u$ and $\varepsilon$ but not on $\gamma$, such that for sufficiently large $x$, $$\P \bigg( D > C \gamma x^{-1} \exp \big( (3 \pi^2)^{1/3} (t - s)^{1/3} \big) \bigg) \leq 6 \varepsilon.$$
\end{Lemma}

\begin{proof}
Let $A = 2 \gamma$, and let $r = s - A x^2$.  For $u \in [s - \gamma x^2, s]$ define $X_u(1)$ in the same way as $X(1)$, but with $u$ playing the role of $s$.  That is, $X_u(1)$ consists of the number of particles at time $u$ whose ancestor was positioned to the left of $L(u)$ at time $v$ for all $v \in [r, u]$.  By the argument leading to Lemma \ref{fphi},
\begin{equation}\label{NXu}
\P(N(u) = X_u(1) \mbox{ for all }u \in [s - \gamma x^2, s]) > 1 - 3 \varepsilon
\end{equation}
for sufficiently large $x$.  By Lemma \ref{meanXf}, there is a positive constant $C$ such that $\E[X_u(1)|{\cal F}_r] \leq C x^{-2} {\hat Z}$ for sufficiently large $x$, where ${\hat Z}$ is defined as in (\ref{hatZdef}) but with $u$ in place of $s$.  The argument leading to (\ref{ZZhat}) implies that on an event with probability greater than $1 - \varepsilon$, we have $\E[X_u(1)|{\cal F}_r] \leq C x^{-2} Z(r)$ for all $u \in [s - \gamma x^2, s]$ for sufficiently large $x$, where $C$ is some other positive constant.

Define times $s - \gamma x^2 = u_0 < u_1 < \dots < u_j = s$, where the $u_i$ are chosen such that $1/2 \leq u_i - u_{i-1} \leq 1$ for $i = 1, 2 \dots, j$.  For $i = 0, 1, \dots, j-1$, let $D_i$ be the number of particles that are killed at the origin between times $u_i$ and $u_{i+1}$.  Let $D_i'$ be the number of such particles that are descended from particles at time $u_i$ that are counted in $X_{u_i}(1)$, meaning that their ancestor was positioned to the left of $L(u_i)$ throughout the time period $[r,u_i]$.  Even in the absence of killing between times $u_i$ and $u_{i+1}$, the expected number of descendants at time $u_{i+1}$ produced by a given particle at time $u_i$ is at most $e^{u_{i+1} - u_i} \leq e$.  It follows that for sufficiently large $x$, $$\E[D_i'|{\cal F}_r] \leq e \E[X_{u_i}(1)|{\cal F}_r] \leq C x^{-2} Z(r)$$ for all $i$ on an event of probability at least $1 - \varepsilon$, and therefore, $$\E \bigg[ \sum_{i=0}^{j-1} D_i' \bigg| {\cal F}_r \bigg] \leq C \gamma Z(r)$$ on an event of probability at least $1 - \varepsilon$.  In view of
Proposition \ref{Zupper}, there is a positive constant $C$ such that for sufficiently large $x$, $$\E \bigg[ \sum_{i=0}^{j-1} D_i' \bigg| {\cal F}_r \bigg] \leq  C \gamma x^{-1} \exp \big( (3 \pi^2)^{1/3} (t - r)^{1/3} \big)$$ on an event of probability at least $1 - 2 \varepsilon$.  By Markov's Inequality, there is a positive constant $C$ such that for sufficiently large $x$, $$\P \bigg( \sum_{i=0}^{j-1} D_i' > C \gamma x^{-1} \exp \big( (3 \pi^2)^{1/3} (t - r)^{1/3} \big) \bigg) \leq 3 \varepsilon.$$  Because $\P(D = \sum_{i=0}^{j-1} D_i') > 1 - 3 \varepsilon$ by (\ref{NXu}) and $$\exp\big((3 \pi^2)^{1/3}(t-r)^{1/3}\big) \leq C \exp\big((3 \pi^2)^{1/3}(t-s)^{1/3}\big),$$ the result follows.
\end{proof}

\begin{proof}[Proof of Theorem \ref{rtthm}]
Fix $d \in \R$.  Let $\gamma \in (0,1]$.  Let $r = s - \gamma x^2$.  Let $$p_i = u \bigg( \gamma x^2, L(s) - \frac{3}{\sqrt{2}} \log x + d - X_i(r) \bigg).$$ Let $R(s)$ be the position of the right-most particle at time $s$ for a modified process in which particles that reach the origin between times $r$ and $s$ are not killed.  Then $$\P \bigg(R(s) \geq L(s) - \frac{3}{\sqrt{2}} \log x + d \bigg| {\cal F}_{r} \bigg) = 1 - \prod_{i=1}^{N(r)} (1 - p_i).$$  Therefore,
\begin{equation}\label{rteq}
1 - \exp \bigg(-\sum_{i=1}^{N(r)} p_i \bigg) \leq \P \bigg(R(s) \geq L(s) - \frac{3}{\sqrt{2}} \log x + d \bigg| {\cal F}_{r} \bigg) \leq \sum_{i=1}^{N(r)} p_i.
\end{equation}
Consequently, the key to the proof will be obtaining a precise estimate of $\sum_{i=1}^{N(r)} p_i$.

Note that $$p_i = u \bigg(\gamma x^2,  L(s) - \frac{3}{2\sqrt{2}} \log \gamma x^2 + \frac{3}{2 \sqrt{2}} \log \gamma + d - X_i(r) \bigg).$$  Because $L(r) - L(s)$ is bounded above by a constant depending on $u$, it follows from Lemma \ref{withina} that with probability tending to one as $x \rightarrow \infty$, we have
\begin{equation}\label{goodevent}
X_1(r) \leq L(s) + \frac{3}{2 \sqrt{2}} \log \gamma + d - 1 - C_{12},
\end{equation}
where $C_{12}$ is the constant from (\ref{735}).
By Lemma \ref{u1lem}, on this event for sufficiently large $x$ we have
\begin{equation}\label{pRST}
C' R_i S_i T_i \leq p_i \leq C'' R_i S_i T_i
\end{equation}
for all $i$, where
\begin{align}
R_i &= L(s) + \frac{3}{2 \sqrt{2}} \log \gamma + d - X_i(r), \nonumber \\
S_i &= \exp \bigg( - \sqrt{2} \big( L(s) + (3/2 \sqrt{2}) \log \gamma + d - X_i(r) \big) \bigg), \nonumber \\
T_i &= \exp \bigg( - \frac{(L(s) + (3/2\sqrt{2}) \log \gamma + d - X_i(r))^2}{2 \gamma x^2} \bigg). \nonumber
\end{align}
Let $$a = L(s) - L(r) + \frac{3}{2 \sqrt{2}} \log \gamma + d.$$  Then
\begin{equation}\label{Ri}
R_i = L(r) \bigg(1 - \frac{X_i(r)}{L(r)} + \frac{a}{L(r)}\bigg).
\end{equation}
Also,
\begin{equation}\label{Si}
S_i = \gamma^{-3/2} e^{-\sqrt{2} d} e^{-\sqrt{2} L(s)} e^{\sqrt{2} X_i(r)}.
\end{equation}
Finally, because $$\frac{L(s)^2}{2 \gamma x^2} = \frac{c^2 (t-s)^{2/3}}{2 \gamma c^2 t^{2/3}} = \frac{1}{2 \gamma} \bigg(1 - \frac{s}{t} \bigg)^{2/3} = \frac{1}{2 \gamma} \bigg(1 - \frac{u}{\tau} \bigg)^{2/3},$$ we have
\begin{align}\label{Ti}
T_i &= \exp \bigg(- \frac{1}{2 \gamma x^2} \bigg( \big(L(r) - X_i(r)\big)^2 + 2a (L(r) - X_i(r)) + a^2 \bigg) \bigg) \nonumber \\
&= \exp \bigg( - \frac{L(s)^2 - (L(s)^2 - L(r)^2)}{2 \gamma x^2} \bigg(1 - \frac{X_i(r)}{L(r)} \bigg)^2 - \frac{2a (L(r) - X_i(r)) + a^2}{2 \gamma x^2} \bigg) \nonumber \\
&= \exp \bigg( - \frac{1}{2 \gamma} \bigg(1 - \frac{u}{\tau} \bigg)^{2/3} \bigg(1 - \frac{X_i(r)}{L(r)} \bigg)^2 \bigg) U_i,
\end{align}
where $U_i \rightarrow 1$ as $x \rightarrow \infty$ uniformly in $i$ because $a/x \rightarrow 0$ and $(L(s)^2 - L(r)^2)/x^2 \rightarrow 0$ as $x \rightarrow \infty$.   Therefore, by (\ref{Ri}), (\ref{Si}), and (\ref{Ti}),
\begin{align}\label{RST}
\sum_{i=1}^{N(r)} R_i S_i T_i &= \gamma^{-3/2} e^{-\sqrt{2} d} e^{-\sqrt{2} L(s)} L(r) \sum_{i=1}^{N(r)} U_i e^{\sqrt{2} X_i(r)} \bigg(1 - \frac{X_i(r)}{L(r)} + \frac{a}{L(r)} \bigg) \nonumber \\
&\qquad \times \exp \bigg( - \frac{1}{2 \gamma} \bigg(1 - \frac{u}{\tau} \bigg)^{2/3} \bigg(1 - \frac{X_i(r)}{L(r)} \bigg)^2 \bigg).
\end{align}

Consider the function $\phi: [0,1] \rightarrow \R$ defined by $$\phi(z) = (1 - z) \exp \bigg( - \frac{1}{2 \gamma} \bigg(1 - \frac{u}{\tau} \bigg)^{2/3} (1-z)^2 \bigg).$$  By (\ref{phiprob}), applied with $s_n = ux_n^3 - \gamma x_n^2$, where $(x_n)_{n=1}^{\infty}$ is a sequence tending to infinity, we have,
\begin{equation}\label{Yprob}
\frac{1}{Y(r)} \sum_{i=1}^{N(r)} e^{\sqrt{2}X_i(r)} \phi \bigg( \frac{X_i(r)}{L(r)} \bigg) \rightarrow_p \frac{\pi}{2} \int_0^1 (1 - z) \exp \bigg( - \frac{1}{2 \gamma} \bigg(1 - \frac{u}{\tau} \bigg)^{2/3} (1-z)^2 \bigg) \sin(\pi z) \: dz.
\end{equation}
Now let $\alpha = (2 \gamma)^{-1/2} (1 - u/\tau)^{1/3}$ and make the substitution $y = \alpha (1-z)$ to get that the right-hand side of (\ref{Yprob}) is
\begin{equation}\label{gamasymp}
\frac{\pi}{2} \int_0^{\alpha} \frac{y}{\alpha} e^{-y^2} \sin \bigg( \frac{\pi y}{\alpha} \bigg) \cdot \frac{1}{\alpha} \: dy \asymp \frac{1}{\alpha^3} \asymp \gamma^{3/2},
\end{equation}
where $\asymp$ means that the ratio of the two sides is bounded above and below by positive constants.
Furthermore, $\sum_{i=1}^{N(r)} e^{\sqrt{2}X_i(r)} = Y(r)$ and $a/L(r)$ tends to zero as $x \rightarrow \infty$.  It thus follows from (\ref{RST}), (\ref{Yprob}), and (\ref{gamasymp}) that on the event (\ref{goodevent}), we have
\begin{equation}\label{sumRST}
\sum_{i=1}^{N(r)} R_i S_i T_i = e^{-\sqrt{2} d} e^{-\sqrt{2} L(s)} L(r) Y(r) H(u, x, \gamma),
\end{equation}
where $H(u, x, \gamma)$ converges in probability as $x \rightarrow \infty$ to some number which is bounded between two positive constants that do not depend on $\gamma$.  Note that $e^{-\sqrt{2} L(s)} = e^{-(3 \pi^2)^{1/3}(t - s)^{1/3}}$.  Therefore, because $Z(r) \leq Y(r)$, we can use Propositions \ref{Zlower} and \ref{Yupper} to conclude that with probability at least $1 - 2 \varepsilon$, we have $C' \leq e^{-\sqrt{2} L(s)} L(r) Y(r) \leq C''$ for sufficiently large $x$.  Combining this result with (\ref{pRST}) and (\ref{sumRST}), we get that there are constants $C_{16}$ and $C_{17}$, not depending on $\gamma$, such that for sufficiently large $x$,
\begin{equation}\label{pibound}
\P \bigg( C_{16} e^{-\sqrt{2} d} \leq \sum_{i=1}^{N(r)} p_i \leq C_{17} e^{-\sqrt{2} d} \bigg) > 1 - 3 \varepsilon.
\end{equation}

Now choose $d_2 > 0$ large enough that $C_{17} e^{-\sqrt{2} d_2} < \varepsilon$.  By (\ref{rteq}) and (\ref{pibound}),
\begin{align}\label{bigrt1}
\P \bigg(X_1(s) \geq L(s) - \frac{3}{\sqrt{2}} \log x + d_2 \bigg) &\leq \P \bigg(R(s) \geq L(s) - \frac{3}{\sqrt{2}} \log x + d_2 \bigg) \nonumber \\
&\leq C_{17} e^{-\sqrt{2} d_2} + 3 \varepsilon \nonumber \\
&\leq 4 \varepsilon.
\end{align}
Likewise, we can choose $d_1 > 0$ large enough that $\exp(-C_{16} e^{\sqrt{2} d_1}) \leq \varepsilon$.  By (\ref{rteq}) and (\ref{pibound}),
\begin{equation}\label{Rlower}
\P \bigg(R(s) \leq L(s) - \frac{3}{\sqrt{2}} \log x - d_1 \bigg) \leq \exp \big( -C_{16} e^{\sqrt{2} d_1} \big) + 3 \varepsilon \leq 4 \varepsilon.
\end{equation}
It remains to bound the probability that $R(s) > L(s) - (3/\sqrt{2}) \log x - d_1$ but $X_1(s) \leq L(s) - (3/\sqrt{2}) \log x - d_1$.  This could only happen if some particle reaches $0$ between times $r$ and $s$ and then, for the modified process in which killing is suppressed during this time, some descendant particle is to the right of $L(s) - (3/\sqrt{2}) \log x - d_1$ at time $s$.  However, by Lemma \ref{Dlem}, with probability at least $1 - 6 \varepsilon$, at most $C \gamma x^{-1} \exp((3 \pi^2)^{1/3} (t-s)^{1/3}) = C \gamma x^{-1} e^{\sqrt{2} L(s)}$ particles reach the origin between times $r$ and $s$.  Conditional on this event, by Lemma \ref{u2lem}, the expected number of these particles with a descendant to the right of $L(s) - (3/\sqrt{2}) \log x - y$ at time $s$ is at most $$C \gamma x^{-1} e^{\sqrt{2} L(s)} \cdot \gamma^{-3/2} x^{-3} L(s) e^{-\sqrt{2}(L(s) - (3/\sqrt{2}) \log x - d_1)} e^{-C_{15}/\gamma} \leq C_{18} \gamma^{-1/2} e^{\sqrt{2} d_1} e^{-C_{15}/\gamma}.$$  Combining this result with (\ref{Rlower}) and Markov's Inequality, and choosing $\gamma$ small enough that $C_{18} \gamma^{-1/2} e^{\sqrt{2} d_1} e^{-C_{15}/\gamma} < \varepsilon$, we get, for sufficiently large $x$,
\begin{equation}\label{bigrt2}
\P \bigg( X_1(s) \leq L(s) - \frac{3}{\sqrt{2}} \log x - d_1 \bigg) \leq 4 \varepsilon + 6 \varepsilon + C_{18} \gamma^{-1/2} e^{\sqrt{2} d_1} e^{-C_{15}/\gamma} \leq 11 \varepsilon.
\end{equation}
The result follows from (\ref{bigrt1}) and (\ref{bigrt2}).
\end{proof}


\bigskip \noindent
Julien Berestycki: \\
Universit\'e Pierre et Marie Curie.  LPMA / UMR 7599, Bo\^ite courrier 188. 75252 Paris Cedex 05

\medskip \noindent Nathana\"el Berestycki: \\
DPMMS, University of Cambridge. Wilberforce Rd., Cambridge CB3 0WB

\medskip \noindent
Jason Schweinsberg:\\
University of California at San Diego, Dept. of Mathematics. 9500 Gilman Drive; La Jolla, CA 92093-0112

\end{document}